\documentclass[reqno]{amsart}
\usepackage{amsmath,amsthm,amssymb}
\usepackage{graphicx,color}
\usepackage{hyperref}
\usepackage{mathrsfs}
\usepackage[utf8]{inputenc}
\usepackage{wasysym}
\usepackage{dsfont}
\usepackage{tikz-cd}
\usepackage[all]{xy}

\definecolor{Red}{cmyk}{0,1,1,0}

\definecolor{Blue}{cmyk}{1,1,0,0}

\makeatletter
\newtheorem*{rep@theorem}{\rep@title}
\newcommand{\newreptheorem}[2]{%
	\newenvironment{rep#1}[1]{%
		\def\rep@title{#2 \ref{##1}}%
		\begin{rep@theorem}}%
		{\end{rep@theorem}}}
\makeatother

\theoremstyle{plain}
\newtheorem{theorem}{Theorem}[section]
\newtheorem{corollary}[theorem]{Corollary}
\newtheorem{proposition}[theorem]{Proposition}
\newtheorem{lemma}[theorem]{Lemma}

\theoremstyle{definition}
\newtheorem{definition}[theorem]{Definition}
\newtheorem{remark}[theorem]{Remark}
\newtheorem{example}[theorem]{Example}

\newreptheorem{theorem}{Theorem}
\newreptheorem{lemma}{Lemma}
\newreptheorem{proposition}{Proposition}
\newreptheorem{corollary}{Corollary}

\usepackage[textsize=scriptsize,color=green!20,shadow]{todonotes}
\presetkeys{todonotes}{fancyline}{}
\newcommand{\hooklongrightarrow}{\lhook\joinrel\longrightarrow}
\DeclareMathOperator*{\esssup}{ess\,sup}

\makeatletter
\@namedef{subjclassname@2020}{\textup{2020} Mathematics Subject Classification}
\makeatother

\title[On the Dimension of the Space of Harmonic Functions]
{On the Dimension of the Space of Harmonic Functions on Transitive Shift Spaces}
\author[L. Cioletti]{L. Cioletti}
\address{Department of Mathematics, Universidade de Bras\'ilia, 70910-900, Bras\'ilia, Brazil}
\email{cioletti@mat.unb.br}
\author[L. Melo]{L. Melo}
\address{Department of Mathematics, Universidade de Bras\'ilia, 70910-900, Bras\'ilia, Brazil}
\email{melo@mat.unb.br}
\author[R. Ruviaro]{R. Ruviaro}
\address{Department of Mathematics, Universidade de Bras\'ilia, 70910-900, Bras\'ilia, Brazil}
\email{ruviaro@mat.unb.br}
\author[E. A. Silva]{E. A. Silva}
\address{Department of Mathematics, Universidade de Bras\'ilia, 70910-900, Bras\'ilia, Brazil}
\email{e.a.silva@mat.unb.br}
\thanks{This work was partially supported by the Coordena\c c\~ao de
	Aperfei\c coamento de Pessoal de N\'ivel Superior - Brasil (CAPES) - Finance Code 001.
	L. Cioletti, R. Ruviaro and E.~A. Silva acknowledge financial support by FAP-DF. L. Cioletti
	is supported by CNPq through project PQ 313217/2018-1. L. Melo acknowledges
  financial support by Minist\'erio da Economia - Brasil.
}
\date{}

\allowdisplaybreaks

\begin{document}

\begin{abstract}
	In this paper, we show a new relation between phase transition 
	in one-dimensional Statistical Mechanics and the multiplicity 
	of the dimension of the space of harmonic functions for an extension 
	of the classical transfer operator. 
	We accomplish this by extending the classical Ruelle-Perron-Frobenius 
	theory to the realm of low regular potentials. This is done by 
	establishing finer properties of the associated conformal measures 
	and thoroughly developing a method to obtain information on the 
	maximal eigenspace of a suitably constructed family of Markov Processes. 
	Our results are valid in the setting of finite and infinite alphabets. 
	Several new applications are given to illustrate the theory. 
	For example, we determine the support of a large class of 
	equilibrium states associated with low regular potentials, 
	including ones allowing phase transition. 
	Additionally, we prove a version of the Functional Central Limit Theorem for 
	equilibrium states. A remarkable aspect of this result is that 
	it does not require the spectral gap property of the associated transfer operator. 
	It is valid for long-range spins systems that might not be positively correlated and
	for non-local observables. 
\end{abstract}
\vspace*{-2cm}
\maketitle

\setlength{\marginparwidth}{3.8cm}
\vspace*{-1cm}

\tableofcontents

\section{Introduction}

This paper is about the Perron-Frobenius eigenvector space of a class of transfer 
operators arising in Ergodic Theory and Equilibrium Statistical Mechanics. 
For this class of operators, defined by a general continuous potential, 
our main result relates the dimension of this subspace with the number 
of extreme conformal measures. Consequently, we show that the geometric multiplicity 
of this eigenvector space can only be greater than one if a 
first-order phase transition takes place, as in 
Dobrushin \cite{MR0250631} and Lanford and Ruelle \cite{MR256687}.

In the one-sided full shift, 
the Perron-Frobenius eigenvector space simplicity is, in general, linked to the absence of phase transition in Statistical Mechanics \cite{MR0234697}. 
Such results have been proven in several other works by exploring 
the regularity properties of the potential associated with the transfer operator, 
see \cite{MR1793194,MR2423393,MR3377291,MR1085356,MR2342978,MR0466493}. 
To a certain extent, it is fair to say that,  to date, the techniques employed 
to study this problem boil down to establishing the pre-compactness 
of a particular family of functions, usually obtained by a variation of the 
Arzel\`a-Ascoli Theorem. 
Here we are mostly interested in the case where 
the potential is continuous but does not have these regularity properties. 
In such cases, little is known about the Perron-Frobenius eigenvector space. 
The quest for a theory that can handle such a general class of potentials 
is the primary motivation of the present article. 
To develop this theory, we introduce a new set of tools that help us 
understand the behavior of the maximal spectral data of transfer operators 
associated with low regular potentials and their relation with the phase transition phenomenon.

A function belonging to this Perron-Frobenius eigenvector space 
is sometimes called a harmonic function \cite{MR1078079,MR1837681,MR4043220}. 
In the present article, this terminology is further motivated 
by the conclusion of our Theorem \ref{teo-positivo}. 
Roughly speaking, it says that a function on this space that vanishes on 
a set of non-trivial measure has to be identically zero. 

In Ergodic Theory, the spectral analysis of transfer operators is a useful tool to understand the asymptotic behavior of complicated non-linear dynamical systems. 
In Statistical Mechanics, these operators play a major role in computing 
the asymptotic behavior of partition functions of lattice spin systems. 
Consequently, the transfer operator carries relevant information on either 
the presence or absence of phase transitions when the temperature varies. 
Since the literature on both subjects and their relationship is vast, 
we refer the reader to 
\cite{MR1793194,MR3793614,MR2423393,MR4007162,MR1014246,MR3852182,MR3701541,
	MauldinUrbanski:2003,MR576928,MR1085356,MR0234697,MR2129258,MR1738951,MR0233038,
	MR412389,MR0466493,MR1783787,MR2342978} 
and references therein for more details. 

In symbolic dynamics, which is the context of this work, the transfer 
operator $\mathscr{L}$ is usually constructed from a continuous potential $f$, 
which is simply a continuous function from some subset of a metrizable product 
space $X$ to $\mathbb{R}$. Generally speaking, in Ergodic Theory the potential $f$ 
is responsible for encoding information on the local rates of expansion 
or contraction of the dynamical system.  In Statistical Mechanics, it is related to the 
intensity of the micro-local interactions between particles. 
In both settings, the transfer operator helps us understand the passage 
from local to global properties, such as Gibbs Measures, topological pressure, 
asymptotics of Birkhoff sums, and so on. 

\break

\subsection{Notation and Statements of the Main Results}\label{sec:preliminaries}

In this subsection, we set up some necessary notations 
before presenting the statements of all our main results. 
Along with the summary of this paper's main results, we  
remark on the hypotheses and provide comparisons with similar results in the literature. 

\medskip

Let $(E,d_{E})$ be a compact metric space,
$X=E^{\mathbb{N}}$ the product space equipped with
a metric $d$ inducing the product topology, and $\mathscr{B}(X)$
the Borel sigma-algebra on $X$.
As usual,  $(C(X),\|\cdot\|_{\infty})$ denotes the Banach space of all
real-valued continuous functions on $X$, endowed with the supremum norm.
The set of all extended-real-valued $\mathscr{B}(X)$-measurable
functions on $X$ is denoted here by
$M(X,\mathscr{B}(X))$.
For an arbitrary metric space $Y$ let us denote by
$\mathscr{M}_{s}(Y), \mathscr{M}(Y)$, and $\mathscr{M}_{1}(Y)$ the space
of all finite Borel signed, positive and probability measures on $Y$, respectively.

The class of transfer operators we deal with in this paper are
those constructed in the following way.
We fix a continuous potential $f:X\to\mathbb{R}$, and an \textit{a priori} Borel probability
measure $p\in \mathscr{M}_{1}(E)$. The operator $\mathscr{L}:C(X)\to C(X)$ sends
a continuous function $\varphi$ to a continuous function
$\mathscr{L}\varphi$, which is defined for any $x\in X$, as 
follows:
\begin{align}\label{def-Lf}
	\mathscr{L}\varphi(x)
	\equiv
	\int_{E} \exp(f(ax))\varphi(ax)\, dp(a),\quad \text{where}\ ax\equiv (a,x_1,x_2,\ldots).
\end{align}

\medskip

Tychonof's theorem implies that $X$ is a compact space.
And, by the Riesz-Markov representation theorem,
the Banach transpose of the transfer operator,
denoted by  $\mathscr{L}^{*}$, can be identified with
a bounded  linear operator acting on $\mathscr{M}_{s}(X)$.
The transpose operator, $\mathscr{L}^{*}$, is a positive operator and takes the cone
of all finite Borel positive measures to itself.
Furthermore, by using the compactness of $\mathscr{M}_{1}(X)$
with respect to  the weak-$*$-topology
and a classical argument based on the Tychonov-Schauder Theorem,
we can show that the
set of eigen-probabilities associated with the spectral radius
of $\mathscr{L}$, i.e.,
\[
\mathscr{G}^{*}\equiv \{\nu\in \mathscr{M}_{1}(X) :  \mathscr{L}^*\nu =\rho(\mathscr{L})\nu \}
\]
is always non-empty (see \cite{CLS20}).

Denker and Urbanski \cite{MR1014246} 
call a probability measure $\nu\in\mathscr{G}^*$ a 
$\rho(\mathscr{L}_{f})$-conformal measure. Here, we follow a similar terminology
except measure $\nu\in\mathscr{G}^*$ is simply called a conformal measure.   
The underlying dynamics here
is always the left shift-map 
$\sigma:X\to X$ given by $(x_1,x_2,x_3,\ldots)\longmapsto (x_2,x_3,x_4,\ldots)$.

We work with the extension of the transfer operator \eqref{def-Lf} to  
the Lebesgue space $L^{1}(\nu)\equiv L^{1}(X,\mathscr{B}(X),\nu)$, 
where $\nu\in \mathscr{G}^{*}$.
We need to be careful when talking about
the $L^1(\nu)$-extension of the transfer operator in \eqref{def-Lf}. 
In our setting, where uncountable alphabets are allowed, 
there is generally no guarantee that $L^{1}(\nu)$ is larger than 
$C(X)$. More precisely, depending on the support of $\nu$, notation $\mathrm{supp}(\nu)$, 
there might be no linear embedding $\pi:C(X)\hooklongrightarrow L^1(\nu)$. 
For instance, if $E$ is an uncountable set and the support 
of the {a priori} measure ($\mathrm{supp}(p)$) is a finite
set, such embedding does not exist. We deal with this technical issue in the appendix
by proving that if the {a priori} measure $p$ is fully supported, then any $\nu\in \mathscr{G}^{*}$ 
is fully supported. This, in turn, implies the existence of a linear embedding $\pi:C(X)\hooklongrightarrow L^1(\nu)$.  

Our main results are presented in Section \ref{sec:markov}. 
There, we consider the appropriate assumptions to construct an extension $\mathbb{L}$
of our original transfer operator $\mathscr{L}$. By taking a suitable normalization,  
the extension $\mathbb{L}$ has operator norm $1$ and
can be viewed as a Markov process. 
Next, we develop the basic theory for such processes, taking into account our 
Markov operators' particularities. 
Our first remarkable result is the following theorem:

\begin{reptheorem}{teo-positivo}
	Let $f$ be any continuous potential and $\nu\in \mathrm{ex}(\mathscr{G}^{*})$ (extreme point).
	If $u\geqslant 0$ is a harmonic function for $\mathbb{L}:L^1(\nu)\to L^1(\nu)$,
	then $u>0$ $\nu$-a.e.
\end{reptheorem}

We apply this theorem to get the support of equilibrium measures 
for low regular potentials in Subsection \ref{sec-app-EqSt} under mild assumptions. 
It states that if we can construct a harmonic function that is positive in 
at least one set of positive measures, then this property has to hold globally. 
This result opens doors to new local techniques to construct 
fully supported equilibrium states in a partially expanding setting. 

We later use the above theorem to prove the following result:

\begin{reptheorem}{teo-unica}
	If $\nu$ is an extreme point in $\mathscr{G}^{*}$
	then the space of harmonic functions (SHF) for the operator $\mathbb{L}:L^1(\nu) \rightarrow L^1(\nu)$,
	has dimension at most one.
\end{reptheorem}

The above result has two fascinating consequences. 
The first one is an application
on the study of phase transition 
in Equilibrium Statistical Mechanics, presented in 
Subsection \ref{sec-phase-transition}. 
The second one is given by Corollary \ref{cor-dim-continuas}. 

At the end of Section \ref{sec:markov}, we present the theorem 
that is one of the main results of this paper:

\begin{reptheorem}{th:eigenspace-dimension}
	Let $f$ be an arbitrary continuous potential, $p:\mathscr{B}(E)\to[0,1]$ a fully supported probability
	measure,
	and  $m\in\mathscr{G}^{*}$ an arbitrary conformal measure.
	Thus, the eigenspace of	$\mathbb{L}:L^1(m)\to L^1(m)$ associated to its
	spectral radius must have a dimension no bigger than
	the cardinality of the set of extreme points in $\mathscr{G}^{*}$.
\end{reptheorem}

The proof of Theorem \ref{th:eigenspace-dimension} reveals that SHF's dimension is 
maximized when the conformal measure $m$ is taken as the barycenter 
of $\mathscr{G}^*$ (in the sense of Choquet \cite{MR80720}). 
We also remark that the result of Theorem \ref{th:eigenspace-dimension} is optimal in the 
sense that there is a potential $f$ for which the space of harmonic function is trivial
and the upper bound is saturated, as shown in the examples of Sections \ref{sec-Curie-Weiss} and \ref{sec-g-measures}.

The mechanism behind the multiplicity of the space of harmonic
functions here is different from all the ones described above. 
Let us expand on this comment. 
The comparison with non-transitive systems is obvious.
However, in the other cases, where the dynamics are forward transitive, 
the comparison reveals the new phenomenon discovered here. 

Note that the setting in \cite{MR2887917} is analogous to ours. 
There, the transfer operator, up to switching from the interval 
$[0,1]$ to the symbolic space $X=\{0,1\}^{\mathbb{N}}$, is given by 
\begin{equation*}
	\mathscr{L}(\psi)(x)= \sum_{a\in E}\psi(ax)\varphi(ax), \quad \forall x \in X,
\end{equation*}
where $\psi$ is fixed and plays the role of the potential, 
and $\varphi$ is a continuous test function, $E=\{0,1\}$. 
As observed earlier, the main difference from our context is that $\psi$ is allowed to have zeros.
This is crucial in \cite{MR2887917}, since the dimension of 
the SHF for $\mathscr{L}$ is bounded by the number of disjoint closed subsets of $X$
where transitive subsystems can be defined.
In our case, however, $\psi=\exp(f)$ and is always positive.
Moreover, our underlying dynamical system is transitive and the space
of continuous harmonic functions for $\mathscr{L}:C(X)\to C(X)$,
has dimension at most one, as proved in Corollary \ref{cor-dim-continuas}.

However, we can make an analogy between the operator $\mathscr{L}$ in \cite{MR2887917} - 
acting on the space of continuous functions - and the extension $\mathbb{L}$ to $L^1(\nu)$
of our transfer operator. 
Both have the dimension of the SHF's bounded 
by the number of disjoint sets that are invariant
by appropriate dynamics. In \cite{MR2887917}, the dynamics is the doubling mapping. 
Here, the dynamics is not the underlying dynamics, 
the left shift mapping, but the dynamics given by 
the homogeneous discrete time Markov process defined on the phase space $X$. 
In our case, each of the invariant sets have full measure with respect 
to the extreme points in $\mathscr{G}^*$,
which implies that they are dense in $X$, since the conformal measures are
fully supported (Theorem \ref{Teo-EP-fully-supp}).
This distinction is essential, as the SHF's
for $\mathscr{L}$ are unaffected by phase transitions.
The opposite is true for the extension $\mathbb{L}$, since its
SHF is spanned by independent harmonic functions supported on
disjoint invariant sets, when they exist.
This phenomenon is exemplified by the discussion in Subsection \ref{sec-Curie-Weiss} and \ref{sec-g-measures}, 
where the multiplicity of the SHF emerges from 
finer properties of the associated Markov process.

\medskip

The reason for developing this theory for the full shift 
is not just for the sake of mathematical generalization.
In fact, the full shift is the only one-sided shift suitable 
for describing one-dimensional one-sided spin systems on Equilibrium Statistical Mechanics
as shown in \cite{CLS20}. Also, the techniques developed here 
provide new insights into the physics of the phase transition problem, 
and identify new mathematical phenomenon, which is the multiplicity of the 
Perron-Frobenius eigenvector space in the forward transitive setting.
Additionally, this theory provides further applications of Thermodynamic Formalism and  
is a simpler description of the phase transition in one-dimensional Statistical Mechanics through 
the dynamics (not the underlying dynamics $\sigma:X\to X$) 
determined by the Markov process introduced in Section \ref{sec:markov}.   

\medskip   

Continuing with our main results, we mention the ones appearing 
in the Applications and Examples section. More precisely,
in Subsection \ref{sec-conf-full-supp}, we prove that any conformal measure 
in $\mathscr{G}^{*}$ has full support. In the subsequent subsection, 
we use this result to study the support of equilibrium states for a 
very general class of potentials and conclude that they are also fully supported.
As far as we know, this is the first general result for equilibrium states that is
valid even in cases where the potential allows for a first-order phase transition. 
We also remark that the needing of a fully supported
conformal measure is present in some previous works related to 
transfer operators defined with {a priori} measure on uncountable compact 
metric alphabets as in \cite{CLS20,lopes2020information,MR3377291}. 
Nevertheless, in these papers, 
the authors did not directly address this question; instead, they assumed 
this property without any further discussion on its validity. 

\bigskip

In Subsection \ref{FCLT}, we use the extension of the transfer operator to $L^2(\nu)$, 
associated with a normalized potential, to show the validity of a 
version of the Functional Central Limit Theorem (FCLT)
for observables solving the Poisson equation. More precisely, we prove the following 
theorem:
\begin{reptheorem}{functionalclt}
Let $P$ be the transfer operator induced by the extension $\mathbb{L}$
associated with a continuous
and normalized potential and $\mu\in \mathscr{G}^{*}$.
Let $\phi:X\to \mathbb{R}$ be a non-constant observable  in $L^{2}(\mu)$
satisfying $\mu(\phi)=0$. If
there exists a solution $\upsilon\in L^{2}(\mu)$ for Poisson's equation $(I-\mathbb{L})\upsilon=\phi$,
then the stochastic process $Y_{n}(t)$, given by
\begin{align*}
Y_{n}(t)=\displaystyle\frac{1}{\varrho\sqrt{n}}\sum_{j=0}^{[nt]}\phi\circ \sigma^{j}, \qquad 0\leqslant t<\infty,
\end{align*}
where $\varrho=\mu(\upsilon^2)-\mu(P\upsilon^2)$, converges in distribution to the Wiener measure in
$D[0,\infty)$.
\end{reptheorem}

In Example \ref{ex-no-spe-gap}, we consider a continuous but not H\"older continuous potential $f$.
We show that the associated transfer operator $\mathscr{L}$ acts on the space $C^{\beta\log}$
(see definition of this space in Example \ref{ex-no-spe-gap}).
The interest of this example lies on the fact that the restriction 
$\mathscr{L}|_{C^{\beta\log}}$ possesses a maximal positive eigenfunction but 
does not have the spectral gap property. 
By taking a sufficiently large $\beta$ and applying Theorem \ref{functionalclt}, we prove
for any observable $\phi\in C^{\beta \log}$, the existence of $\varrho>0$ such that
the stochastic process $Y_{n}(t)$  
converges in distribution to the Wiener measure in $D[0,\infty)$.

\subsection{Harmonic Functions on Non-transitive Settings}\label{Sec-HistBackground}

Multidimensional\break spaces of harmonic functions (SHF) for a 
transfer operator have appeared in several applications: 
On functional equations related to ergodic theory and Markov chains 
\cite{MR1078079}; on specific hyperbolic maps with metastable states \cite{MR2887917,MR2832249}; 
on multiresolution wavelet theory \cite{MR1837681,MR1008470}; and on the 
Ruelle-Perron-Frobenius Theorem in the context of non-forward topologically 
transitive subshifts of finite type \cite{MR1793194,MR1005524}.

Here, the multiplicity of SHF emerges from a mechanism that is entirely 
different from the ones listed above. 
Firstly, in contrast to the aforementioned works, our dynamic $\sigma:X\to X$ is forward transitive. 
In \cite{MR1793194,MR1078079,MR2887917,MR2832249,MR1837681,MR1008470}, 
linearly independent harmonic functions are supported 
on disjoint closed invariant sets, with respect to their respective dynamics. 
Furthermore, in these works 
the multiplicity of SHF is solely encoded by their underlying dynamical systems. 
In our case, the harmonic functions are also supported on invariant sets. 
However, the multiplicity is encoded in the potential. 
The invariant sets here have a completely different description since 
they are all dense sets on the phase space $X$. 
Moreover, their multiplicity emerges from the phase transition phenomenon, 
which is related to the regularity properties of the potential. 
The complications brought to the low regularity of the potentials 
are overcome by studying the structure of these invariant 
sets through a suitable Markov process constructed from the transfer operator. 
As we will see, the structure of the invariant sets is determined 
by the invariant sets of the process, which in turn are 
determined by the support of 
the extreme measures in the convex set $\mathscr{G}^{*}$.
Next, we provide further details on the settings
considered in the work mentioned previously.

In an abstract setting of functional equations, Conze and Raugi studied in a seminal paper 
\cite{MR1078079} SHFs for the following transfer operator, 
whose action on a test function  $\varphi:[0,1]\to\mathbb{R}$ is defined by 
\[
\mathscr{L}\varphi (x) 
= 
u(x/2)\varphi(x/2) 
+
u((x+1)/2)\varphi((x+1)/2),
\]
where $u:[0,1]\to [0,1]$ is a fixed continuous function (playing a similar role as a $g$-function)
satisfying $u(x/2)+u(x/2+1/2)=1$, 
for all $x\in [0,1/2]$. 
Their analysis is based on the properties of a martingale constructed from the iterates $\{\mathscr{L}^n\}_{n\in\mathbb{N}}$. An important feature in their work 
is that the function $u$ is allowed to vanish.
The authors also proved that the dimension of the SHF for 
$\mathscr{L}$ is bounded by the number of disjoint closed subsets of 
$X$ where transitive subsystems can be defined.

Dolgopyat and Wright \cite{MR2887917} studied metastable systems. They began with a 
hyperbolic interval map $T_0:I\to I$ with $m$ 
disjoint invariant sub-intervals $I_1, I_2,\ldots, I_m$.
Their map has $m$ mutually singular ergodic absolutely continuous 
invariant measures (ACIMs) $\mu_1,\mu_2,\ldots, \mu_m$,
which determine $m$ linearly independent harmonic
functions for the transfer operator $\mathscr{L}:\mathrm{BV}(I)\to BV(I)$:
\[
\mathscr{L}_{0}\varphi(x) = \sum_{y\in T^{-1}(x)} \frac{\varphi(y)}{|T_{0}'(y)|}.
\]
They constructed a metastable system $T_{\varepsilon}$ by perturbing, 
in a special way, the initial map $T_{0}$ and approximating its unique ACIM $\mu_{\varepsilon}$
by a convex combination of $\mu_1,\mu_2,\ldots,\mu_m$. After the $\varepsilon$-perturbation,
the dynamics become transitive and the SHF for the transfer operator
$\mathscr{L}_{\varepsilon}$ turned one-dimensional. 

Jorgensen \cite{MR1837681} considered a very general class of transfer
operators inspired by the multiresolution wavelet theory. 
Its guiding example was given by the transfer operator described below. 
Let $N\geqslant 2$ be a fixed integer and consider the Lebesgue space $L^1(\mathbb{T})\equiv L^1(\mathbb{T},\mathscr{B}(\mathbb{T}),\lambda)$ 
where $\mathbb{T}\equiv \{z\in \mathbb{C}: |z|=1\}$ and $\lambda$ is its unique Haar measure.
Thus the operator is defined by
\[
\mathscr{L}\varphi(z) = \frac{1}{N} \sum_{w^N=z} |m_0(w)|^2\varphi(w), 
\]
where $m\in L^{\infty}(\mathbb{T})$.
In this context it is usual to consider a further normalizing 
condition such as $\mathscr{L}\mathds{1}=\mathds{1}$ (quadrature wavelet filters). 
Next, the author determined the dimension of the space of harmonic 
functions by computing the number of a specific class of normal 
representations of the $C^*$-algebra $\mathfrak{A}_{N}$ on two 
unitary generators $U$ and $V$ satisfying the relation $UVU^{-1}=V^N$.

In subshifts of finite type, multidimensional SHF shows up 
in Theorem 1.5 in \cite{MR1793194}. To be more precise, let $N\geqslant 2$ be a fixed integer number, 
$E=\{1,\ldots,N\}$, and $A$ an $N\times N$ matrix with coefficients in $\{0,1\}$. The subshift
of finite type defined by the transition matrix $A$ is the restriction $\sigma_{A}^{+}$
of the full shift $\sigma^{+}: X\to X$ (on the one-sided product space $X=E^{\mathbb{N}}$) 
to the invariant set $\Sigma_{A}^{+}\equiv \{x\in X: A(x_n,x_{n+1})=1, \ \forall n\in \mathbb{N}\}$.
For any $0<\theta<1$, consider the metric $d_{\theta}$ on $\Sigma_{A}^{+}$ defined by the expression 
$d_{\theta}(x,y) = \sum_{n\in\mathbb{N}} \theta^n \delta_{K}(x_n,y_n)$, where $\delta_{K}$
is the Kroenecker delta.  
Let $f:\Sigma_{A}^{+} \to \mathbb{R}$ be a non-negative H\"older potential and consider 
the transfer operator $\mathscr{L}:C(\Sigma_{A}^{+})\to C(\Sigma_{A}^{+})$ given by
\[
\mathscr{L}\varphi (x) = \sum_{y\in (\sigma_{A}^{+})^{-1}(x)} e^{f(y)}\varphi(y).
\]
Assume that 
\[
\rho(\mathscr{L},C(\Sigma^{+}_{A})) \equiv \lim_{n\to\infty} 
\Big( \sup_{x\in \Sigma^{+}_{A}} \mathscr{L}^n\mathds{1}\Big)^{\frac{1}{n}}\neq 0.
\]
Then $\rho(\mathscr{L},C(\Sigma^{+}_{A}))$ is an eigenvalue and it can have a geometric multiplicity $m\geqslant 1$
if $\sigma_{A}^{+}$ is not forward transitive. Under these assumptions, item (3.b) of
Theorem 1.5 in \cite{MR1793194} shows that there is a basis $\{h_1,\ldots, h_{m}\}$ of the eigenspace
of $\mathscr{L}$ associated to the spectral radius $\rho(\mathscr{L},C(\Sigma^{+}_{A}))$. Actually,
they also show there are probability measures $\{\nu_1,\ldots, \nu_{m}\}$ satisfying 
$\mathscr{L}^{*}\nu_{j} = \rho(\mathscr{L},C(\Sigma^{+}_{A})) \nu_{j}$,
and $\langle \mu_{j},h_{i}\rangle = \delta_{K}(i,j)$, for all $j=1,\ldots m$.

\subsection{Organization of the Paper}
In Section \ref{sec:markov}, we use  Hopf's theory 
of Markov processes to prove our main results. 
In Section \ref{sec-exp-app}, we present some examples and applications 
of the theory developed in the previous section. It is divided into seven subsections.   
In Subsection \ref{sec-app-EqSt}, we study the support of some equilibrium states
associated to low regular continuous potentials.  
In Subsection \ref{sec-Curie-Weiss}, 
we introduce the Curie-Weiss potential and show that the SHF 
associated to this model is two-dimensional. 
This example aims to illustrate the abstract theory developed in Section \ref{sec:markov}. 
In Subsection \ref{sec-g-measures}, we provide
several examples where the dimension of SHF is greater than one.
In Section \ref{sec-suff-cond-eigenfunctions},
we present some classical results, adapted to our setting, that ensures the existence 
of harmonic functions. In Subsection \ref{sec-phase-transition} we provide an application in
Equilibrium Statistical Mechanics. 
We show that the SHF dimension is related to the phase transition 
phenomenon in one-dimensional lattice spin systems.
In Section \ref{FCLT}, we prove the validity of 
an FCLT by using the Poisson equation. We use our theorem to obtain an FCLT for 
the Dyson model (in the uniqueness regime) with a non-local observable. 
In Section \ref{sec-concluding-rmk}, we present concluding remarks. 

In the appendix, we develop some technical results of 
the support of the conformal measures and the existence of the extension of the transfer operator,
and we compute the spectrum of the extension. 
The main results are: Theorems \ref{Teo-EP-fully-supp} and \ref{Teo-Extension-L1}, and Proposition
\ref{mu_f satisfies C1} and \ref{prop-spectrum}. Although these results are well-known 
in finite alphabet settings, $|E|<+\infty$, they are fundamental 
and new in more general settings such as uncountable compact alphabets.

\section{Harmonic Functions and Markov Processes}\label{sec:markov}

Throughout this section, $f$ is a continuous potential, $\nu\in\mathscr{G}^{*}$ is a conformal measure,
$\mathscr{L}:C(X)\to C(X)$ is the transfer operator defined by \eqref{def-Lf}, 
and the operator $\mathbb{L}:L^1(\nu)\to L^1(\nu)$ is the extension of $\mathscr{L}$ provided by Theorem \ref{Teo-Extension-L1}.
Up to adding a constant to the potential, we can always assume that $\rho(\mathscr{L})=1$. 
Therefore, $\mathbb{L}$  can be seen as a Markov process in the Hopf's sense \cite{hopf},
which
means that it is a positive contraction on $L^1(\nu)$. Stating it
more precisely:

\begin{definition}[Hopf-Markov Processes]
A Markov process is defined as an ordered quadruple $(X,\mathscr{F},\mu,T)$, 
where the triple $(X,\mathscr{F},\mu)$ is a sigma-finite measure space with a positive measure
$\mu$ and $T$ is a bounded linear operator acting on $L^1(\mu)$
satisfying:
\begin{itemize}
	\item[(1)] $T$ is a contraction: 
	$\sup\{\|T\varphi\|_{1}: \|\varphi\|_{1}\leqslant 1 \}\equiv \|T\|_{\mathrm{op}}\leqslant 1$;
	\item[(2)] $T$ is a positive operator, that is, if $\varphi\geqslant 0$, then $T\varphi\geqslant 0$.
\end{itemize}
\end{definition}

Here, the sigma-algebra $\mathscr{F}$ will be the Borel sigma-algebra $\mathscr{B}(X)$, 
$T$ is the extension $\mathbb{L}:L^1(\nu)\to L^1(\nu)$ of 
the transfer operator $\mathscr{L}$, and $\mu=\nu$. 
Condition (2), the positivity property of $\mathbb{L}$,
is inherited from $\mathscr{L}$, and the condition (1) follows from  $\|\mathbb{L}\|_{\mathrm{op}}=\rho(\mathscr{L})$ (Theorem \ref{Teo-Extension-L1}),
and the assumption $\rho (\mathscr{L})=1$.

Since $\nu\in\mathscr{G}^{*}$, we have that
\begin{equation}\label{eigencondition}
\int_X \mathscr{L} \varphi\, d\nu = \int_X \varphi \, d\nu \quad \forall \varphi \in C(X,\mathbb{R}).
\end{equation}
This duality relation can be rewritten as $\langle \mathds{1}, \mathbb{L}\varphi \rangle_\nu = \langle
\mathds{1},
\varphi \rangle_\nu$ for every $\varphi \in C(X,\mathbb{R})$, where 
$\langle \cdot, \cdot \rangle_\nu$
is the usual bilinear form which puts $L^\infty(\nu)$ (left entry) and $L^1(\nu)$ (right
entry) in duality, and $\mathds{1} \in L^\infty(\nu)$ is the $\nu$-equivalence class
of the constant function equal to one.

From the continuity of $\mathbb{L}$ and the density of $C(X)$ on $L^1(\nu)$, we get that
\begin{equation*}
\langle \mathbb{L}^*\mathds{1},  u \rangle_\nu = \langle \mathds{1},  \mathbb{L}u \rangle_\nu
= \langle \mathds{1},  u \rangle_\nu \quad \forall u \in L^1(\nu).
\end{equation*}
This means that $\mathbb{L}^*\mathds{1} = \mathds{1}$ and therefore 
$\mathds{1}$ is always an eigenfunction of its dual.
On the other hand, the existence of a positive or non-negative eigenfunction
for $\mathbb{L}$ itself
is a much more delicate issue. In Subsection \ref{sec-suff-cond-eigenfunctions},
we discuss this problem
and provide some necessary and sufficient conditions for its existence and uniqueness.

\subsection{Measurable Identity Principle}

In the general theory of Markov processes, a measurable set $B\in\mathscr{B}(X)$
satisfying $\mathbb{L}^{*}\mathds{1}_{B}=\mathds{1}_{B}$ 
is sometimes called an invariant set for the process. 
Such sets play an essential role in the theory to be developed ahead and 
we begin by showing that conditioning a conformal measure to an invariant set 
results again in a conformal measure.  

\begin{proposition}\label{prop-conditional-coformal}
If $\mathbb{L}^*\mathds{1}_B = \mathds{1}_B$, for some  $B\in\mathscr{B}(X)$,
then the Borel measure $\nu_{B}$ given by $A\longmapsto \nu(A\cap B)$ is an eigenmeasure for $\mathscr{L}^{*}$.
Moreover, if $\nu(B)\neq 0$ then the conditional measure $A\longmapsto \nu(A\cap B)/\nu(B)\equiv \nu(A|B)$
is an element of $\mathscr{G}^*$.
\end{proposition}

\begin{proof}
To prove the first statement it is enough to use the condition $\mathbb{L}^*\mathds{1}_B = \mathds{1}_B$
together with \eqref{eigencondition}.  Indeed, for any continuous function $\varphi$ we have
\begin{align*}
\int_{X}\mathscr{L}\varphi\, d\nu_{B}
 & =  \int_{X}\mathds{1}_{B} \, \mathscr{L}\varphi  \, d\nu
= \langle\mathds{1}_{B},\mathscr{L}\varphi \rangle_{\nu}
\\
 & =\langle\mathds{1}_{B},\mathbb{L}\varphi \rangle_{\nu}
=\langle\mathbb{L}^{*}\mathds{1}_{B},\varphi \rangle_{\nu}
\\
 & =
\langle \mathds{1}_{B},\varphi \rangle_{\nu}
=
\int_{X} \varphi\, d\nu_{B}.
\end{align*}

If $\nu(B)\neq 0$ we have  immediately that
$\nu(\cdot\cap B)/\nu(B) = \nu(\cdot|B)$
is a conformal measure.
\end{proof}

\begin{lemma}\label{lema-extreme-01-law}
If $\nu\in\mathscr{G}^{*}$ is an extreme point, then there is no $B\in \mathscr{B}(X)$
such that $0<\nu(B)<1$ and $\mathbb{L}^{*}\mathds{1}_{B}=\mathds{1}_{B}$.
\end{lemma}

\begin{proof}
Suppose, by contradiction,
that such a Borel set $B$ does exist. From Proposition \ref{prop-conditional-coformal}
we know that $\nu(\cdot|B)$ is a conformal measure.
Since $\mathbb{L}^*\mathds{1} = \mathds{1}$, linearity of $\mathbb{L}^{*}$
implies that $\mathbb{L}^*\mathds{1}_{B^c}=\mathds{1}_{B^c}$. By applying again
Proposition \ref{prop-conditional-coformal} we get that
$\nu(\cdot|B^c)$ is a conformal measure. Clearly, $\nu(\cdot|B)\neq \nu(\cdot|B^c)$.
But $\nu = \nu(B)\nu(\cdot|B)+\nu(B^c)\nu(\cdot|B^c)$
which contradicts the assumption that $\nu$ is extreme.
\end{proof}

The following result shows that any harmonic function for $\mathbb{L}$ 
satisfies a kind of identity principle. 
More precisely, it says that if a non-negative  harmonic function vanishes
on a
set of positive $\nu$-measure ($\nu\in \mathrm{ex}(\mathscr{G}^{*})$), then it should vanish
$\nu$-almost everywhere.

\begin{theorem}\label{teo-positivo}
Let $\nu$ be an extreme point in $\mathscr{G}^{*}$
and $u\geqslant 0$  a harmonic function (therefore not identically zero) of
$\mathbb{L}: L^1(\nu) \rightarrow L^1(\nu)$, associated
to its operator norm.  Then $u>0$ $\nu$-a.e.
\end{theorem}

\begin{proof}
Suppose, by contradiction, there is a set $B\in \mathscr{B}(X)$
such that $0<\nu(B)<1$, $u|_B=0$ and $u_{B^c}>0$.
Since $\mathbb{L}u=u$, we get that
\begin{equation}\label{eq:L*}
\langle \mathbb{L}^* \mathds{1}_{B^c}, u \rangle_{\nu}
=
\langle \mathds{1}_{B^c}, \mathbb{L}u \rangle_{\nu}
=
\langle \mathds{1}_{B^c}, u \rangle_{\nu}.
\end{equation}
Note that $\mathbb{L}^*1_{B^c} \leqslant 1$,
because the adjoint of a positive contraction is
also a positive contraction. 
Since  $u$ is non-negative and supported on $B^c$, and
$0\leqslant \mathbb{L}^{*}\mathds{1}_{B^c} \leqslant 1$,
it follows from \ref{eq:L*} that
$\mathds{1}_{B^c} \mathbb{L}^*\mathds{1}_{B^c}=\mathds{1}_{B^c}$.
From these observations, we get that
$\mathbb{L}^*\mathds{1}_{B^c}\geqslant \mathds{1}_{B^c}$,
since $\mathbb{L}^{*}$ is positive. Therefore
\begin{equation*}
\lVert \mathbb{L}^* \mathds{1}_{B^c} \rVert_1
\geqslant
\lVert \mathds{1}_{B^c} \rVert_1.
\end{equation*}

As we already mentioned, $\mathbb{L}^{*}$ is a contraction with respect to the $L^\infty(\nu)$-norm.
Moreover, the operator $\mathbb{L}^*$ acts as a contraction, with respect to the $L^1(\nu)$-norm,
on the linear manifold spanned by the characteristic functions. Indeed,
\begin{equation*}
\lVert \mathbb{L}^* \mathds{1}_{B^c} \rVert_1
= \langle \mathbb{L}^* 1_{B^c}, \mathds{1}  \rangle_\nu
= \langle \mathds{1}_{B^c}, \mathbb{L} \mathds{1}  \rangle_\nu
\leqslant
\langle \mathds{1}_{B^c}, \mathds{1}  \rangle_\nu
=
\lVert \mathds{1}_{B^c} \rVert_1.
\end{equation*}

Since
$0\leqslant \mathds{1}_{B^c}\leqslant  \mathbb{L}^{*}\mathds{1}_{B^c}$
and
$\lVert \mathbb{L}^* \mathds{1}_{B^c}
\rVert_1 = \lVert \mathds{1}_{B^c} \rVert_1$,
it follows that the equation
$\mathbb{L}^*\mathds{1}_{B^c}=\mathds{1}_{B^c}$ holds.
On the other hand, $0<\nu(B^c)<1$ and $\nu$ is an extreme point in $\mathscr{G}^{*}$
and so Proposition \ref{prop-conditional-coformal} applies and we get a contradiction.
\end{proof}

In the next subsection, we show that this set of ideas can also be used
to handle the eigenspace's simplicity associated with the eigenvalue one, 
when the conformal measure is an extreme point in $\mathscr{G}^{*}$.
Before, we need the following lemma.

\begin{lemma}\label{lema-eigenfunctions-positivas}
If $\nu$ is an extreme point in $\mathscr{G}^{*}$
and $u\in L^1(\nu)$ is a harmonic function of $\mathbb{L}$ associated to one,
then $u$ has a definite sign $\nu$-almost everywhere.
\end{lemma}

\begin{proof}
The proof is by contradiction.
Assume $u$ has a non-trivial decomposition on its  positive and negative parts, that is,
$0<\nu(\{x\in X: u(x)> 0\})<1$ and $0<\nu(\{x\in X: u(x)\leqslant 0\})<1$.
We call $B=\{x\in X: u(x)> 0\}$. By using the linearity of $\mathbb{L}$ and $\mathbb{L}u=u$, we get
\begin{equation}\label{L u^+}
\begin{split}
\mathbb{L} (u^+) 
&= \mathbb{L} (u + u^-)
= u + \mathbb{L}(u^-)
= u^{+} + (\mathbb{L} u^{-} - u^{-})
\\
&= (u^{+} + (\mathbb{L} u^{-} - u^{-}))\mathds{1}_B  + (u^{+} + (\mathbb{L} u^{-} - u^{-}))\mathds{1}_{B^c}
\\
&=(u^{+} +\mathbb{L} u^{-} )\mathds{1}_{B} +(\mathbb{L} u^{-} - u^{-}))\mathds{1}_{B^c}.
\end{split}
\end{equation}
By multiplying both sides of \eqref{L u^+} by $\mathds{1}_{B^c}$,
we get from the positivity of $\mathbb{L}$
that $0\leqslant \mathds{1}_{B^c}\mathbb{L} (u^+)  = (\mathbb{L} u^{-} - u^{-}))\mathds{1}_{B^c}$.
Therefore
\begin{align*}
0	\leqslant \int_{X} (\mathbb{L} u^{-} - u^{-}))\mathds{1}_{B^c}\, d\nu
 & =
\int_{X} \mathds{1}_{B^c}\mathbb{L} u^{-} - u^{-} \, d\nu
\\
 & \leq
\int_{X}\mathbb{L} u^{-}\, d\nu  - \int_{X} u^{-} \, d\nu
\\
 & =
\|\mathbb{L} u^{-}\|_{1} - \|u^{-}\|_{1}\leqslant 0,
\end{align*}
where in the last inequality we used the contraction property of $\mathbb{L}$.
This shows that $(\mathbb{L} u^{-} - u^{-}))\mathds{1}_{B^c}=0$ $\nu$-a.e..
Replacing this in \eqref{L u^+} we get the equality
$\mathbb{L}u^{+} =(u^{+} +\mathbb{L} u^{-} )\mathds{1}_{B}$.
Now, we integrate both sides of this equality obtaining
$\|\mathbb{L}u^{+}\|_{1} =\|u^{+}\|_{1} +\|\mathbb{L} u^{-}\mathds{1}_{B}\|_{1}$.
Applying once more	the contraction property  we have that
$\|\mathbb{L} u^{-}\mathds{1}_{B}\|_{1}=0$. This implies $\mathbb{L}u^{-}\mathds{1}_{B}=0$
$\nu$-a.e.. Finally, from the identity $\mathbb{L}u^{+} =(u^{+} +\mathbb{L} u^{-} )\mathds{1}_{B}$,
it follows that $\mathbb{L}u^{+} =\mathds{1}_{B} u^{+} = u^{+}$.

Since we are assuming that $\nu$ is an extreme point in $\mathscr{G}^{*}$ and we have shown that
$u^{+}$ is a not identically zero non-negative harmonic function of $\mathbb{L}$, we can apply
Theorem \ref{teo-positivo} to get that	$u^{+}>0$ $\nu$-a.e., which implies that $u=u^{+}$,
contradicting the assumption that  $u$ has non-trivial positive and negative parts.
\end{proof}

\subsection{The Dimension of the Space of Harmonic Functions}

\begin{theorem}\label{teo-unica}
If $\nu$ is an extreme point in $\mathscr{G}^{*}$
then the dimension of the eigenspace associated to
the operator norm of $\mathbb{L}:L^1(\nu) \rightarrow L^1(\nu)$
is at most one.
\end{theorem}

\begin{proof}
The proof is by contradiction. Suppose that there are two linearly independent harmonic functions
$u$ and $v$ for $\mathbb{L}$. By Lemma \ref{lema-eigenfunctions-positivas}, we can assume that they
are both positive almost everywhere.
Let $\int_X u d\nu \equiv \nu(u)$ be the mean value of $u$, with respect to $\nu$.
Consider the following linear combination $w \equiv u - (\nu(u)/\nu(v))v$.
Since $w$ is a linear combination of two
harmonic functions follows that it is also a harmonic function.
Clearly, $\nu(w)=0$, and so it is either identically zero or
it has non-trivial positive and negative parts.
Note that it can not be identically zero since $w$ is a linear
combination of two linear independent functions.
Nor $w$ can have a positive and a negative part,
since it would contradict Lemma
\ref{lema-eigenfunctions-positivas} and therefore we have a contradiction.
\end{proof}

Now we have the tools to prove the following Corollary.
\begin{corollary}\label{cor-dim-continuas}
Let $f$ be any continuous potential and $\mathscr{L}:C(X)\to C(X)$ be a transfer operator constructed from
this potential
and a fully supported a priori probability measure $p$ on $E$. Then the eigenspace of $\mathscr{L}$
associated to its
spectral radius has either dimension zero or one.
\end{corollary}
\begin{proof}
Up to adding a constant to the potential $f$, we can assume that $\rho(\mathscr{L})=1$.
Since $X$ is a compact metric space we have that $\mathscr{G}^{*}$ is convex and compact.
By Krein-Milman Theorem we have that the set of extreme points of $\mathscr{G}^{*}$ is necessarily not empty.
Take any $\nu\in \mathrm{ex}(\mathscr{G}^{*})$.
Suppose that $\varphi$ and $\psi$ are two linearly independent
continuous harmonic functions for $\mathscr{L}$,
associated to the eigenvalue one. Note that we are in conditions to apply Theorem \ref{Teo-EP-fully-supp}
and therefore we have that the probability measure $\nu$ is fully supported.
From this fact follows that $[\varphi]_\nu \not = [\psi]_\nu$.
Indeed, they are linearly independent in $L^1(\nu)$.

Now, we consider  the extension $\mathbb{L}:L^1(\nu)\to L^1(\nu)$
of the transfer  operator $\mathscr{L}$ provided by Theorem \ref{Teo-Extension-L1}.
Of course, $[\varphi]_\nu$ and $[\psi]_\nu$
are harmonic functions of $\mathbb{L}$, associated to the eigenvalue one.
Then Lemma \ref{lema-eigenfunctions-positivas}
ensure that  they have definite signs, which can be assumed to be  positive.
But Theorem \ref{teo-unica} implies $[\varphi]_\nu= \lambda [\psi]_\nu$,
for some real number $\lambda$,
which is a contradiction.
\end{proof}

We will end this section presenting a result on the problem of the dimension of the maximal eigenspace
of $\mathbb{L}:L^1(m) \rightarrow L^1(m)$, when $m$ is as before a conformal measure,
but not necessarily an extreme point in $\mathscr{G}^{*}$. Nevertheless, first, we need
the following lemma.

\begin{lemma}\label{lema-conditional-coformal-reciprocal}
Let  $\nu, \mu \in \mathrm{ex}(\mathscr{G}^*)$ be distinct measures
and $m$ a non-trivial convex combination of them, $m=t\nu+(1-t)\mu$.
Then there is a $\mathscr{B}(X)$-measurable set, $B$, for which
$\nu(B)=1$, $\mu(B^c)=1$ and the following equations hold
\begin{equation}
\begin{split} \label{eq:L*-converse}
\mathbb{L}^*\mathds{1}_B=\mathds{1}_B \quad \text{ and } \quad \mathbb{L}^*\mathds{1}_{B^c}=\mathds{1}_{B^c},
\end{split}
\end{equation}
where $\mathbb{L}$ is the extension of $\mathscr{L}$ to $L^1(m)$.
Moreover if $D\in\mathscr{B}(X)$ is another set satisfying  $\mathbb{L}^{*}\mathds{1}_{D}=\mathds{1}_{D}$ and
$0<m(D)<1$,
then either $m(D\bigtriangleup B)=0$ or $m(D\bigtriangleup B^c )=0$,
where $D\bigtriangleup B$ denotes the symmetric difference between $D$ and $B$.
\end{lemma}

\begin{proof}
In the appendix of reference \cite{CLS20} the authors adapted 
Theorem 7.7 item (c) of \cite{MR2807681}
to our setting.
This result  says that any extreme point
in $\mathscr{G}^*$ is uniquely determined by
the values it takes on the elements of
the tail sigma-algebra $\mathscr{T}$.
Since $\nu$ and $\mu$ are distinct and determined by the values
taken on $\mathscr{T}$, there is at least one element
$B \in \mathscr{T}$
such that $\mu(B) \not = \nu(B)$.
But from Corollary 10.5 in \cite{CLS20} we
also know that any extreme point in $\mathscr{G}^*$
is trivial on $\mathscr{T}$,
meaning that, for every $B\in \mathscr{T}$, $\nu(B)=0$ or $\nu(B)=1$.
Then, supposing that $\mu(B)=0$, we have that $\nu(B)=1$.
We now have two disjoint sets, $B$ and $B^c$ with
$\nu(B)=1$ and $\mu(B^c)=1$.

Following the computations of  Proposition \ref{prop-conditional-coformal},
we will show that $\mathbb{L}^*\mathds{1}_B=\mathds{1}_B$.
Actually, it is enough to prove that
$\langle \mathbb{L}^*\mathds{1}_B, \varphi \rangle_m = \langle \mathds{1}_B,
\varphi \rangle_m$
for every $\varphi \in C(X,\mathbb{R})$.
Indeed, for an arbitrary continuous function $\varphi$ we have
\begin{align*}
\langle \mathbb{L}^*\mathds{1}_B, \varphi \rangle_m
 & = \langle \mathds{1}_B, \mathbb{L}\varphi \rangle_m
= \langle \mathds{1}_B, \mathscr{L}\varphi \rangle_m
= \int_{X}\mathds{1}_B \mathscr{L}\varphi\, dm         \\
 & =  \int_{X} \mathscr{L}\varphi \, dm(\cdot\cap B)
=  \int_X \mathscr{L}\varphi \, d(t\nu)
= \int_X \varphi \, d(t\mathscr{L}^*\nu)               \\
 & = \int_X \varphi \, d(t\nu)
= \int_X \mathds{1}_B \varphi \, dm
= \langle \mathds{1}_{B},\varphi \rangle_m
\end{align*}
The third equality above holds because $m(\cdot\cap B) = t\nu$.
Again by using the density of $C(X)$ in $L^1(m)$,
we conclude that $\mathbb{L}^*\mathds{1}_B=\mathds{1}_B$.
Recalling that	$\mathbb{L}^{*}\mathds{1}=\mathds{1}$
we get from previous identity that $\mathbb{L}^{*}\mathds{1}_{B^c}=\mathds{1}_{B^c}$.

It remains to prove  the $m$-almost everywhere uniqueness of $B$. More precisely, suppose that
there exist $D\in\mathscr{B}(X)$ with $0<m(D)<1$ such that $\mathbb{L}^{*}\mathds{1}_{D}=\mathds{1}_{D}$
and $m(D\bigtriangleup B)>0$.
Then we have to show that $m(D\bigtriangleup B^c)=0$.
\begin{figure}[h]
\centering
\includegraphics[width=4.6cm,height=2.5cm]{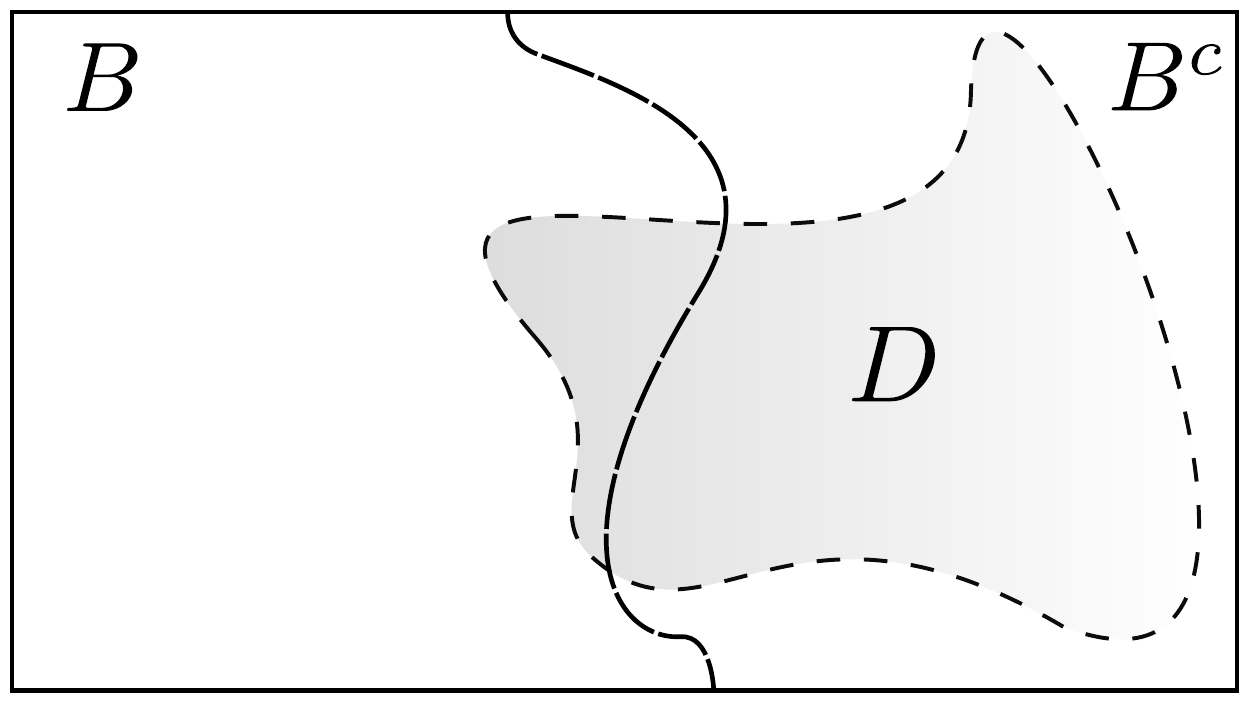}
\label{fig-lemma-symm-diff}
\caption{The sets $B$ and $D$ in the general case.}
\end{figure}

It follows from our assumption that  $m(D\cap B^c)>0$ or $m(D^c\cap B)>0$.
The analysis of both cases are similar and so we can assume that $m(D\cap B^c)>0$.

We claim that $\mathbb{L}^{*}\mathds{1}_{D\cap B^c} = \mathds{1}_{D\cap B^c}$.
Indeed, $\mathds{1}_{D} = \mathbb{L}^{*}(\mathds{1}_{D}\mathds{1}_{B} + \mathds{1}_{D}\mathds{1}_{B^c})$. By
multiplying both sides by $\mathds{1}_{B^c}$
we get
$\mathds{1}_{D}\mathds{1}_{B^c}
=
\mathds{1}_{B^c} \mathbb{L}^{*}(\mathds{1}_{D}\mathds{1}_{B})
+
\mathds{1}_{B^c} \mathbb{L}^{*}(\mathds{1}_{D}\mathds{1}_{B^c})
$.
To prove the claim it is enough to show that
$\mathds{1}_{B^c} \mathbb{L}^{*}(\mathds{1}_{D}\mathds{1}_{B})=0$ and
$\mathds{1}_{B^c} \mathbb{L}^{*}(\mathds{1}_{D}\mathds{1}_{B^c}) =
\mathbb{L}^{*}(\mathds{1}_{D}\mathds{1}_{B^c})$.
These two statements follow  immediately from the positiveness of $\mathbb{L}^{*}$.
In fact,
$0
\leqslant
\mathds{1}_{B^c} \mathbb{L}^{*}(\mathds{1}_{D}\mathds{1}_{B})
\leqslant
\mathds{1}_{B^c}\mathbb{L}^{*}(\mathds{1}_{B})
=
\mathds{1}_{B^c}\mathds{1}_{B} =0.
$

Since $m(D\cap B^c)>0$ it follows from definitions of $m$, $B$ and $t-1>0$
that $0<\mu(D\cap B^c)\leqslant 1$.
We claim that $\mathbb{L}_{\mu}^{*}\mathds{1}_{D\cap B^c}=\mathds{1}_{D\cap B^c}$
and $\mathbb{L}_{\nu}^{*}\mathds{1}_{D\cap B}=\mathds{1}_{D\cap B}$.
Once the claim is established, we can use Lemma \ref{lema-extreme-01-law} to ensure
that   $\mu(D\cap B^c)= 1$, because the measure of this set has to be positive. The equality $\nu(D\cap B)=0$
follows from the fact that it could be zero or one. If it were one, we would have immediately $m(D)=1$, which
is a contradiction.
This information, together with the definition of $B$	and
elementary properties of probability measures,
implies what we wanted to show
$
m(D\bigtriangleup  B^c)
= t\nu(D^c\cap B^c)+(1-t) \mu(D^c\cap B^c)
+ t\nu(D\cap B)+(1-t) \mu(D\cap B)=0
$.

Now, we prove the last claim. Since both equalities have similar proof it is enough
to prove the first one, that is,
$\mathbb{L}_{\mu}^{*}\mathds{1}_{D\cap B^c}=\mathds{1}_{D\cap B^c}$.
Indeed, we already know that
$\mathbb{L}^{*}\mathds{1}_{D\cap B^c}=\mathds{1}_{D\cap B^c}$.
From the results in Section \ref{sec-int-representations} we have that
$L\mathds{1}_{D\cap B^c}(x) = \mathds{1}_{D\cap B^c}(x)$
for almost all $x\in X$, with respect to $m$.  Since $\mu\ll m$ we get
the same conclusion, of the last equation, but for almost all $x$, with respect to $\mu$.
By applying the results of Section \ref{sec-int-representations} we get that
$\mathbb{L}_{\mu}^{*}\mathds{1}_{D\cap B^c}=\mathds{1}_{D\cap B^c}$,
thus finally completing the proof.
\end{proof}

\begin{remark}\label{remark:multidimensional}
Consider the multidimensional  case and $m=\sum_{i=1}^n t_i\nu_i$, where
for each $i\in 1,\ldots,n$, $t_i\in (0,1)$,  $\sum_{i=1}^n t_i=1$ and $\nu_i \in
\mathrm{ex}(\mathscr{G}^{*})$ are distinct conformal measures.
By applying Lemma \ref{lema-conditional-coformal-reciprocal}
we can find a Borel set $B_{1i}$ such that $\nu_1(B_{1i})=1$
and $\nu_i(B_{1i}^c)=1$. Set $B_1=\cap_{i=2}^n B_{1i}$,
it is clear that $\nu_1(B_1)=1$. On the other
hand, since $B_1^c = \cup_{i=2}^n B_{1i}^c \supseteq B_{1i}^c$
and $\nu_i(B_{1i}^c)=1$, we have $\nu_i(B_1^c)=1$, for
all $i \in 2,\ldots,n.$
By considering the set $B_1$ and repeating the arguments in the proof of
Lemma \ref{lema-conditional-coformal-reciprocal}
we get that $\mathbb{L}^*\mathds{1}_{B_1}=\mathds{1}_{B_1}$
and $\mathbb{L}^*\mathds{1}_{B_1^c}=\mathds{1}_{B_1^c}$. Its almost surely uniqueness
is obtained similarly.
\end{remark}

\begin{theorem} \label{th:eigenspace-dimension}
Let  $f$ be an arbitrary continuous potential and $p:\mathscr{B}(E)\to[0,1]$ a fully supported probability
measure,
and  $m\in\mathscr{G}^{*}$ be an arbitrary conformal measure.
Then the eigenspace of	$\mathbb{L}_{m}$ associated to its
spectral radius has dimension not bigger than
the cardinality of the set of extreme points in $\mathscr{G}^{*}$.
\end{theorem}
\begin{proof}
The arguments in this proof involve, simultaneously, different extensions
of transfer operator $\mathscr{L}:C(X)\to C(X)$. To avoid confusions these extensions
will be indexed by the conformal measure as in notation $\mathbb{L}_\nu$.
It has the advantage of let clear on which Lebesgue  space the extension acts.

As before, there is no loss of generality in assuming that $\rho(\mathscr{L})=1$.
Therefore for each conformal measure $m,\nu,\mu\in \mathscr{G}^{*}$ we have that the extensions
$\mathbb{L}_m$,
$\mathbb{L}_\nu$ and $\mathbb{L}_\mu$, provided by Theorem \ref{Teo-Extension-L1},
define themselves Markov processes.

Of course, there is nothing to prove if $\#\mathrm{ex}(\mathscr{G}^*)=+\infty$,
thus in what follows we assume that the cardinality of the set of
extreme points of $\mathscr{G}^{*}$ is finite.

In case $\mathrm{ex}(\mathscr{G}^{*})$ is a singleton
the conclusion follows immediately from Theorem \ref{teo-unica}.
In the sequel we will assume that $\#\mathrm{ex}(\mathscr{G}^{*}) =2$.
The generalization of the following argument
to the case of a convex combination
of a finite number of extreme measures
is straightforward and involves the
application of Remark \ref{remark:multidimensional}.
It is omitted to avoid an unnecessary
cumbersome notation.

We denote by $\mathbb{L}_{m}$ be the extension of $\mathscr{L}:C(X)\to C(X)$,
provided by Theorem \ref{Teo-Extension-L1},  to $L^1(m)$, corresponding to
the measure $m = t\nu + (1-t)\mu$.
Lemma \ref{lema-conditional-coformal-reciprocal}
implies that there is a unique (modulo-$m$) set $B\in \mathscr{B}(X)$
such that $\nu(B)=1$, $\mu(B^c)=1$,
$\mathbb{L}^{*}_m \mathds{1}_B = \mathds{1}_B$
and $\mathbb{L}^{*}_m \mathds{1}_{B^c} = \mathds{1}_{B^c}$.

\medskip

Note that one of the following	three possibilities occurs:
\begin{itemize}
\item[\textit{i})] the eigenvalue problem  $\mathbb{L}_{m}[u]_{m}=[u]_{m}$ has only the trivial solution,
      i.e., $[u]_{m}=0$;
\item[\textit{ii})] any harmonic function $[u]_{m}$ for $\mathbb{L}_{m}$ is such that
      $[\mathds{1}_{B}u]_{m}\neq 0$, but\break $[\mathds{1}_{B^c}u]_{m}= 0$ and vice-versa;
\item[\textit{iii})] there is a harmonic function $[u]_{m}$ such that  both $[\mathds{1}_{B}u]_{m}\neq
      0$, and\break
      $[\mathds{1}_{B^c}u]_{m}\neq 0$.
\end{itemize}
Of course, in the first case the dimension of the maximal eigenspace is zero and the theorem is proved.
We will show next that in the second case, the
maximal eigenspace is one-dimensional. In this case we will  say that the harmonic functions are
supported on either $B$ or $B^c$, depending on where $u$ does not vanish.
Finally,  in the third case we will show that the maximal eigenspace is spanned by two linearly independent
functions $\{[\mathds{1}_{B}u]_{m},[\mathds{1}_{B^c}u]_{m}  \}$, and therefore will be a two-dimensional
space subspace
of $L^1(m)$.

Let us assume that \textit{iii}) holds. We are choosing to handle this case firstly because
the arguments involved in it work similarly in case $\textit{ii})$.

We are going to show that if $[v]_{m}$ is any other harmonic function
then $[v]_{m} = \alpha [\mathds{1}_{B}u]_{m}+\beta [\mathds{1}_{B^c}u]_{m}$, for some
$\alpha,\beta \in\mathbb{R}$.

Firstly, we will show that both $[\mathds{1}_{B}u]_{m}$ and $[\mathds{1}_{B^c}u]_{m}$
are two linearly independent harmonic functions of $\mathbb{L}_{m}$.
The linear independence of these two functions is obvious.
Lets us show that $[\mathds{1}_{B}u]_{m}$ is a harmonic function of $\mathbb{L}_{m}$.
Note that
\begin{equation}\label{eq:split projection}
\begin{split}
\mathbb{L}_{m}[\mathds{1}_{B}u]_{m} = & \mathbb{L}_{m}[u]_{m} - \mathbb{L}_{m}[\mathds{1}_{B^c}u]_{m}
= [u]_{m} - \mathbb{L}_{m}[\mathds{1}_{B^c}u]_{m} \\
= & [\mathds{1}_{B}u]_{m} + [\mathds{1}_{B^c}u]_{m}- \mathbb{L}_{m}[\mathds{1}_{B^c}u]_{m}
\end{split}
\end{equation}
Recalling that
$\mathbb{L}_m^{*}\mathds{1}_{B}=\mathds{1}_{B}$ and using the above equality,
we obtain
\begin{equation*}
\begin{split}
\lVert	\mathds{1}_{B}u \rVert_{L^1(m)}
&=
\langle \mathds{1}_{B}, [\mathds{1}_{B} u] \rangle_m
=
\langle \mathds{1}_{B}, \mathbb{L}_{m}[\mathds{1}_{B} u] \rangle_m
\\
&=
\langle \mathds{1}_{B}, [\mathds{1}_{B}u] + [\mathds{1}_{B^c}u]- \mathbb{L}_{m}[\mathds{1}_{B^c}u] \rangle_{m}
\\
&=
\lVert u\mathds{1}_{B}	\rVert_{L^1(m)} - \langle \mathds{1}_{B}, \mathbb{L}_{m}[\mathds{1}_{B^c}u]
\rangle_m,
\end{split}
\end{equation*}
which implies $\mathbb{L}_{m}[\mathds{1}_{B^c}u]_{m} = 0$ in $B$. Similarly, we get
$\mathbb{L}_{m}[\mathds{1}_{B}u]_{m} = 0$ in $B^c$.
By plugging this back in \ref{eq:split projection} we get that
$\mathbb{L}_{m}[\mathds{1}_{B}u]_{m} = [\mathds{1}_{B}u]_{m}$ and
consequently $\mathbb{L}_{m}[\mathds{1}_{B^c}u]_{m} = [\mathds{1}_{B^c}u]_{m}$.

From definition of $m$,  $\mu(B)=0$, and $\mathbb{L}_{m}[\mathds{1}_{B}u]_m = [\mathds{1}_{B}u]_m$ it follows
that $\mathbb{L}_\nu[\mathds{1}_{B}u]_\nu = [\mathds{1}_{B}u]_\nu$.
The conformal measure $\nu\ll m$ and therefore we get from item \textit{iii}) that $[\mathds{1}_{B}u]_\nu\neq
0$.
Since $\nu\in\mathrm{ex}(\mathscr{G}^{*})$  we can apply  Theorem \ref{teo-positivo} to ensure that
$[\mathds{1}_{B}u]_\nu$ is positive $\nu$-almost everywhere.

Now, let $[v]_{m}$ be an arbitrary harmonic function of $\mathbb{L}_{m}$. By repeating the above
steps we conclude that $[\mathds{1}_{B}v]_{\nu}$ is also a $\nu$-almost everywhere positive harmonic function of
$\mathbb{L}_{\nu}$.
But Theorem \ref{teo-unica} states that there is some $\alpha\in \mathbb{R}$
such that $[\mathds{1}_{B}v]_{\nu}=\alpha [\mathds{1}_{B}u]_{\nu}$.
From the definition of $B$ and $m$ we conclude that
the last equality actually implies $[\mathds{1}_{B}v]_{m}=\alpha [\mathds{1}_{B}u]_{m}$.
By repeating this argument for $[\mathds{1}_{B^c}v]_{\nu}$ we get that
$
[v]_{m}  = \alpha [\mathds{1}_{B}u]_{m}+ \beta [\mathds{1}_{B^c}u]_{m}
$,
which finishes the proof of the theorem.
\end{proof}

\section{Applications and Examples}\label{sec-exp-app}

\subsection{Conformal Measures are Fully Supported}\label{sec-conf-full-supp}

In this section, we prove that any conformal measure $\nu\in\mathscr{G}^{*}$, 
associated with a continuous potential, is fully supported, that is,  $\mathrm{supp}(\nu)=X$ 
if the a priori measure used in the construction of the transfer operator is also fully supported. 
This fact allows us to show that $\mathscr{L}$ extends continuously to $L^1(\nu)$ (Theorem \ref{Teo-Extension-L1}), 
and has some interesting applications, 
as those related to the support of equilibrium states in Subsection \ref{sec-app-EqSt}. 

When one has a finite alphabet, the balls generating the topology 
are also cylinders, and the indicator functions of the cylinders are continuous functions. 
Using this fact, the proof of the full support of the conformal measures is immediate. 
For infinite alphabets, indicator functions of balls are, in general, 
not continuous, and we need to construct an auxiliary continuous function 
to relate the balls in the alphabet $E$ with the balls in the product space $X$.

\begin{theorem}\label{Teo-EP-fully-supp}
	Let $p\in \mathscr{M}_{1}(E)$ be an a priori measure such that $\mathrm{supp}(p)=E$.
	Then for any $\nu\in\mathscr{G}^*$ we have that $\mathrm{supp}(\nu)=X$.
\end{theorem}

\begin{proof}
	Let $x \in X$, $r>0$ and $B(x,r)$ the open ball in $X$ centered in $x$ with radius $r$.
	We will show that $\nu(B(x,r))>0$.
	Since $d$ induces the product topology
	there are $n\in \mathbb{N}$ and $R\in \mathbb{R}$ such that $B(x,r)
	\supseteq B_E(x_1,R) \times ... \times B_E(x_{n},R) \times E^\mathbb{N} \equiv B(R)$.
	
	For each fixed $a\in E$ consider the auxiliary continuous function 
	$\psi_a:E \rightarrow [0,1]$ given by
	\[ \psi_a(b) = \max \left\{ 1 -\frac{2}{R}d_{E}\Big(b,B_E(a,R/2)\Big), 0 \right\}. \]

	Actually, this is a concrete instance of a Urysohn type function.
	Notice that $\psi_a$ is bounded by two characteristic
	functions of two balls centered at $a\in E$, i.e.,
	$ 1_{B_E(a,\frac{R}{2})} \leqslant \psi_a \leqslant 1_{B_E(a,R)}$.
	For each $x\in X$ consider also another auxiliary function $\Psi_x:X\to \mathbb{R}$ given by
	$\Psi_x(y) = \prod_{k=1}^{n} \psi_{x_k}(y_k)$. From the last inequality we get that
	\[
	1_{B(R/2)}(y)\leqslant	\Psi_x(y) \leqslant 1_{B(R)}(y), \qquad \forall y\in X.
	\]
	By using the elementary properties of the transfer operator we get
	\begin{align*}
		\nu(B(x,r))
		& \geqslant \nu(B(R))
		\\
		& = \int_X 1_{B(R)}\,	d\nu
		\geqslant
		\int_X \Psi_x \, d\nu                                                                            \\
		& = \frac{1}{\rho^n(\mathscr{L})}\int_X \Psi_x \, d[{\mathscr{L}^{*}}^n\nu]\,
		= \frac{1}{\rho^n(\mathscr{L})}\int_X  \mathscr{L}^n \Psi_x \, d \nu                             \\
		& = \frac{1}{\rho^n(\mathscr{L})}\int_X \mathscr{L}^{n-1} \left[ \int_E \exp \left(f(a_1 \cdot)
		\right) \Psi(a_1 \cdot)\, dp(a_1) \right](y)
		\, d \nu(y)                                                                                      \\
		& =\frac{1}{\rho^n(\mathscr{L})} \int_X \int_{E^n} \exp
		\left[\sum_{k=1}^n f(a_k\cdots a_n y) \right]
		\prod_{k=1}^{n} \psi_{x_k}(a_k)\, dp^n(a) d\nu(y)                                                \\[0.2cm]
		& \geqslant \left( \frac{\min_{x\in X} \exp(f(x)) }{\rho(\mathscr{L})} \right)^n
		\prod_{k=1}^{n} \left[ \int_{E} \psi_{x_k}(a_k)\, dp(a_k) \right]                                \\[0.2cm]
		& \geqslant
		\left( \frac{\min_{x\in X} \exp(f(x)) }{\rho(\mathscr{L})} \right)^n
		\prod_{k=1}^{n} p(B_E(x_k,R/2)) > 0,
	\end{align*}
	where the existence of a positive minimum follows from the compactness of $X$
	and the continuity of $f$, and $p(B_E(x_k,R/2))>0$ because we are assuming that $p$
	is fully supported.
\end{proof}

\begin{corollary}[Invariant sets are dense in $X$]
Let $p\in \mathscr{M}_{1}(E)$ be an a priori measure such that $\mathrm{supp}(p)=E$, 
$f$ a general continuous potential, $\nu\in\mathscr{G}^*$ and $L:L^1(\nu)\to L^1(\nu)$ the extension of $\mathscr{L}$.
If $B\in\mathscr{B}(X)$ is an invariant set, in the sense that 
$\mathbb{L}^*\mathds{1}_B = \mathds{1}_B$, then $B$ is dense in $X$. 
\end{corollary}

\begin{proof}
	Since we are assuming $B$ is an invariant set with respect to $\mathbb{L}$ 
	it follows from Proposition \ref{prop-conditional-coformal} that 
	$\nu(\cdot | B)$ is also a conformal measure.
	Fix a point $x\in X$. By applying Theorem \ref{Teo-EP-fully-supp} to $\nu(\cdot | B)$ 
	we can conclude that, for every $r>0$, the conditional probability $\nu(B(x,r)|B)>0$. 
	Therefore, at least one point of $B$ is in $B(x,r)$, otherwise $\nu(B(x,r)|B)=0$.
	Since this holds for every $x \in X$ and $r>0$,
	it follows that any invariant set $B$ is dense in $X$.
\end{proof}

\subsection{The Support of Equilibrium States}\label{sec-app-EqSt}

If we want to talk about equilibrium states when working with uncountable alphabets, 
the measure theoretical entropy needs to be 
replaced by another type of entropy in order to have a meaningful definition. 
In this context, there are two interesting alternatives 
generalizing the measure-theoretic entropy. Both   
give rise to concepts of equilibrium states carrying  physical interpretations. 
The first one appeared in  
Statistical Mechanics and it is thoroughly developed in 
\cite{MR2807681,MR517873,MR1241537}. The other one appeared in the Dynamical Systems literature in
reference \cite{MR3377291} in the context of shifts in compact metric alphabets.
Some years later, the authors of \cite{MR3897924} proved that they are actually equivalent to each other. 
For the sake of simplicity, we will adopt the statistical-mechanical way. 
Before presenting its definition, let us remind some of the needed concepts. 
Given $\mu$ and $\nu$ two arbitrary finite measures on $X$ and
$\mathscr{A}$ a sub-$\sigma$-algebra of $\mathscr{F}$, we define 
\[
\mathscr{H}_{\mathscr{A}}(\mu|\nu)
=
\begin{cases}
\displaystyle\int_{X} 
\frac{d\mu|_{\mathscr{A}}}{d\nu|_{\mathscr{A}}} 
\log  \left(\frac{d\mu|_{\mathscr{A}}}{d\nu|_{\mathscr{A}}} \right)
\, \text{d}\nu, &\ \text{if}\ \mu\ll\nu \ \text{on}\ \mathscr{A};
\\[0.5cm]
\infty,&\ \text{otherwise}.
\end{cases}
\]
This is in general a non-negative extended real number, and 
$\mathscr{H}_{\mathscr{A}}(\mu|\nu)$ is called {\it relative entropy} of $\mu$
with respect to $\nu$ on $\mathscr{A}$. 
Let $\pmb{p}=\prod_{i\in\mathbb{N}} p$ be the product measure constructed from our a priori measure $p$.
The entropy we want to consider will be denoted by 
$\mathtt{h}$, and for each shift-invariant probability measure $\mu\in \mathscr{M}_{\sigma}(X)$, 
it is defined as the limit
\[
\mathtt{h}(\mu)\equiv 
-\lim_{n\to\infty} \frac{1}{n} \mathscr{H}_{\mathscr{F}_n}(\mu|\pmb{p}),
\]
where $\mathscr{F}_n$ is the $\sigma$-algebra generated by the 
projections $\{\pi_j:X\to E: 1\leqslant j\leqslant n\}$.
The existence of  such limit follows from a  subadditivity arguments
and the details can be found in \cite{MR2807681}.
Although $\mathtt{h}(\mu)$ is always a non-positive number, 
it is related to the measure theoretical entropy $h_{\mu}(\sigma)$
by the formula $\mathtt{h}(\mu)+\log|E|=h_{\mu}(\sigma)$
when the alphabet $E$ is finite and the a priori measure is taken as the normalized counting
measure. So, both entropies determine the same set of equilibrium states in this particular 
context.  Back to the general case, Proposition 15.14 in \cite{MR2807681}
ensures us that in our context (compact metric alphabets) the mapping 
$\mathscr{M}_{\sigma}(X)\ni \mu\longmapsto \mathtt{h}(\mu)$ is affine and upper semi-continuous,
relative to the weak-$*$-topology, and therefore, for any continuous potential $f$,
there is at least one solution for the generalized version of the variational principle 
\[
\sup_{\mu\in \mathscr{M}_{\sigma}(X)}
\{\mathtt{h}(\mu) + \int_{X} f\, \text{d}\mu\}.
\]

Next, we show how to use the abstract results obtained in the previous section to get information
on the support of the equilibrium states. 
If $E$ is a compact alphabet and $f$ is a sufficiently regular (H\"older, Walters or Bowen) potential, 
then the set of equilibrium states is a singleton
as a consequence of the main result in \cite{MR3897924},
that is, $\mathrm{Eq}(f)=\{\mu\}$. 
Moreover, the unique equilibrium state $\mu$ is obtained in the traditional way by taking 
a suitable scalar multiple of the harmonic function $h$  
for $\mathscr{L}$ and the unique conformal measure in
$\mathscr{G}^{*}=\{\nu\}$, that is, $d\mu=h d\nu$.  

It is natural to expect that this construction also works for general continuous 
potentials that are less regular than those mentioned above. 
This indeed follows from a tedious 
adaptation of the arguments in \cite{MR3897924} or alternatively 
by a suitable extension of the Rokhlin Formula to our context.

Now, we have the following application. Suppose that $f$ is a low regular potential, 
but $\mathscr{G}^{*}$ is a singleton. This last assumption is not so restrictive since
this property holds generically in $C(X)$.   
If a harmonic function $h=\mathbb{L}h$ do exists
then, we can apply Theorem \ref{teo-positivo} to ensure the existence 
of an equilibrium state $\mu\in \mathrm{Eq}$ which is fully supported.
Indeed, the $\nu$-a.e. positivity of $h$
and Theorem \ref{Teo-EP-fully-supp} which ensures that $\mathrm{supp}(\nu)=X$ immediately implies
\[
\mathrm{supp}(\mu)=\mathrm{supp}(h\nu)=X.
\]

Note that this strategy to get a fully supported equilibrium state
still has a weakness which is the existence of the harmonic function $h$.
In general, $h$ is obtained by Functional Analysis arguments
such as the existence of certain invariant cones together with a compactness 
criterion, for example, the Arzel\`a-Ascoli theorem. 
But the extension of the harmonic analysis developed here provides an 
alternative to such methods. In fact, a careful reading of Theorem \ref{teo-positivo}
reveals that if we are able to prove that $\mathscr{L}^n(\mathds{1})(x)$ or it's Ces\`aro mean
can be controlled locally (in a set of positive $\nu$-measure), 
meaning local existence of non-trivial upper and lower bounds, then Theorem \ref{teo-positivo}
ensures that this control is actually global. 
So probability techniques such as sequence of backwards Martingales 
or microlocal-analysis now might play an important role in this problem.

\subsection{Two Dimensional Space of Harmonic Functions}\label{sec-Curie-Weiss}

In this section, we construct an
explicit example of a transfer operator, denoted by $\mathbb{L}_{\beta f}$,
where $\beta>0$ and $f$ is a potential given by \eqref{eq:cw_potencial},
for which its maximal eigenspace is two-dimensional. The potential $f$ and the results in this section are
inspired in the Curie-Weiss (mean-field) model for ferromagnetism.
The aim is to illustrate the results presented in the previous sections in a concrete and simple example,
especially including the construction of a base for the maximal eigenspace of this transfer operator.
There is a special feature of this example. The potential is not  continuous. 
This leads us to introduce a proper replacement for conformal measures. The measures that are going to play the same
role of those in $\mathscr{G}^{*}$, will be called here generalized conformal measures.
This concept is introduced in Definition \ref{def-gen-conf-measure}, and next, 
a motivation is presented.

\bigskip

In what follows $E=\{-1,1\}$, $p = \sum_{e\in E}\delta_{e}$ is the couting measure on $E$ and
$\nu\in\mathscr{M}_{1}(X)$  a probability
measure so that $(p,\nu)$ is a pair satisfying the hypothesis \ref{hip1}.

Fix a bounded and $\mathscr{B}(X)$-measurable potential $f:X\to\mathbb{R}$.
Then the  mapping sending $\varphi\in \mathcal{L}^1(\nu)$ to $L\varphi$ given by
\begin{equation} \label{def-transf-ope-CW-sections}
(L\varphi)(x) = \sum_{a\in E}\exp( f(ax) )\, \varphi(ax), \quad
\text{where} \ ax\equiv (a,x_1,x_2,\ldots)
\end{equation}
defines a linear operator from $\mathcal{L}^{1}(\nu)$ to itself.
Since we are assuming the hypothesis \ref{hip1} we have that the operator $L$
induces a positive and continuous linear operator $\mathbb{L}:L^1(\nu)\to L^1(\nu)$.

\begin{definition}[Generalized Conformal Measures]\label{def-gen-conf-measure}
\label{automedida}
Let $X=E^\mathbb{N}$ be a product space, where $E$ is a finite set and
$f:X \rightarrow \mathbb{R}$ a bounded and
a $\mathscr{B}(X)$-measurable potential.
A measure $\nu\in \mathscr{M}_1(X)$ such that:  $(p,\nu)$ satisfies \ref{hip1};
and the operator $\mathbb{L}$ induced by $L$ satisfies
$\mathbb{L}^*\mathds{1} = \|\mathbb{L}\|_{\mathrm{op}} \mathds{1}$ $\nu$-a.e.
will be called a generalized conformal measure associated to the potential $f$.
\end{definition}

The motivation for this definition comes from the following equivalence on the continuous potential case.
Recall that in Section \ref{sec:preliminaries} we have seen that, for every continuous
potentials $f\in C(X)$,
there is at least one measure for which
$\mathscr{L}^*\nu = \rho(\mathscr{L})\nu$. As mentioned early, this implies
$\mathbb{L}^*\mathds{1} = \rho(\mathscr{L}) \mathds{1}$.
Conversely, Proposition \ref{prop-conditional-coformal}
with $B=X$ shows that
$\mathbb{L}^{*}\mathds{1} = \rho(\mathscr{L})\mathds{1} \implies \mathscr{L}^{*}\nu =\rho(\mathscr{L}) \nu$.
Therefore  for continuous potentials,
\begin{equation*}
\mathscr{L}^{*} \nu = \rho(\mathscr{L})\nu
\quad \Longleftrightarrow \quad
\mathbb{L}^{*}\mathds{1} = \rho(\mathscr{L})\mathds{1} \quad \nu-a. e.
\end{equation*}
This equivalence is the motivation for Definition \ref{def-gen-conf-measure}.

The above definition could be formulated for a general compact metric alphabet $E$,
but this particular form is enough for our needs in this section and avoids some unnecessary technicalities.
We should also remark that although the space $X$ is compact, in this setting there is no guarantee, in general, that
the set of generalized conformal measure is not empty.

\bigskip

From now on, we work with $E=\{-1,+1\}$, $p$ as the counting measure on $E$,
and the potential given by
\begin{equation}
\label{eq:cw_potencial}
f(x) \equiv  x_1\, \limsup_{N\to\infty}  \frac{1}{N}  \sum_{k=2}^{N+1} x_k.
\end{equation}

Clearly, $(\beta f)_{\beta>0}$ is a family of bounded and
$\mathscr{B}(X)$-measurable potentials.
But, differently from the other sections of this paper, every potential in this family is
discontinuous, with respect to the product topology.

The potential defined in \eqref{eq:cw_potencial} is inspired in the mean-field or Curie-Weiss model for ferromagnetism.
It is one of the simplest models in Equilibrium Statistical Mechanics 
exhibiting the phase transition phenomenon,
see \cite{MR3456343,MR503332,MR2189669,MR3982275,CWLW20JSP,MR1008232} 
for more details from the Statistical Mechanics point of view.

Although $f$ is not a continuous potential
it will be useful in exemplifying some concepts of the previous sections.

The following discussion is motivated by some well-known properties of the Curie-Weiss model. 
This discussion aims to convince the reader that specific product measures 
are natural candidates to be generalized conformal measures and to explain 
how to compute the respective eigenvalues, in this case. 

Firstly, we consider the family of product measures $\mu_{\gamma}$, parameterized by $\gamma\in (-1,1)$,
and defined by
\begin{align}\label{def-bernoulli-measure}
\mu_\gamma(\{x_k=+1\}) = p; \quad
\mu_\gamma(\{x_k=-1\}) = 1-p
\end{align}
where the parameter $\gamma= 2p-1$ is the expected value of the coordinate functions, that is,
$\mathbb{E}_{\mu_\gamma}[x_k]=\gamma$, for every $k\in\mathbb{N}$.

Clearly, for any choice of  $\gamma\in (-1,1)$, we have that $(p,\mu_{\gamma})$ satisfies hypothesis \ref{hip1}
and therefore $L_{\beta f}$ induces an operator on $L^1(\mu_{\gamma})$.
Next, we compute the operator norm of $L_{\beta f}$. By definition
\[
\|  \mathbb{L}_{\beta f} \|
=
\int_{X} \sum_{a\in \{-1,1\} } \exp(\beta f(ax)) \, d\mu_{\gamma}(x)
=
\int_{X} \sum_{a\in \{-1,1\}  } \exp(\beta\, a\, m(ax)) \, d\mu_{\gamma}(x)
\]
where
\begin{align}\label{def-mag}
m(x) \equiv \limsup_{N\to\infty} \frac{1}{N}\sum_{k=2}^{N+1} x_k
=
\lim_{N\to\infty} \frac{1}{N}\sum_{k=1}^{N} x_k
= \gamma,  \qquad \mu_\gamma-\text{a.e.}
\end{align}
by the Law of Large Numbers.
Of course, for any $a\in \{-1,1\}$ we have $m(ax)=m(x)$ and therefore
\[
\|  \mathbb{L}_{\beta f} \|
=
\int_{X} e^{-\beta \gamma}+ e^{\beta \gamma} \, d\mu_{\gamma}(x) = 2\cosh(\beta \gamma).
\]
Until now, the parameter $\gamma$ is free, but in order to $\mu_{\gamma}$ to be
a generalized conformal measure for $\beta f$ the equality
$
\langle \mathds{1}, \mathbb{L}_{\beta f}u \rangle_{\mu_\gamma} = 2\cosh(\beta \gamma) \langle \mathds{1},u \rangle_{\mu_\gamma}
$
must holds for any $u\in L^1(\mu_{\gamma})$.
For this, it is enough that,
$\langle \mathds{1}, \mathbb{L}_{\beta f}\mathds{1}_{B} \rangle_{\mu_\gamma} = \cosh(\beta\gamma) \langle \mathds{1},\mathds{1}_{B} \rangle_{\mu_\gamma}$
for any indicator function $\mathds{1}_B$, where
$B\in \mathscr{B}(X)$.

Developing the left-hand side of the last equation we get
\begin{align*}
\int_X \mathbb{L}_{\beta f}\mathds{1}_B\, d\mu_\gamma
 & = \int_X \sum_{a\in\{-1,+1\}} \exp(a\beta m(ax))\mathds{1}_B(ax) \, d\mu_{\gamma}(x)                \\
 & =\int_X  \exp(-\beta\gamma) \mathds{1}_B(-1x) +\exp(\beta\gamma)\mathds{1}_B(+1x) \, d\mu_\gamma(x) \\
 & = \frac{e^{-\beta\gamma}\mu_{\gamma}(B\cap[-1])}{\mu_{\gamma}([-1])}
+\frac{e^{\beta\gamma}\mu_{\gamma}(B\cap [+1])}{\mu_{\gamma}([+1])} .
\end{align*}
By taking $B=[+1]$ and next $B=[-1]$, in the previous identity, and using
$\langle \mathds{1}, \mathbb{L}_{\beta f}\mathds{1}_{B} \rangle_{\mu_\gamma} = \cosh(\beta\gamma) \langle \mathds{1},\mathds{1}_{B} \rangle_{\mu_\gamma}$, we see that the following relations must be satisfied
\begin{align*}
e^{\pm \beta\gamma}
= \int_X \mathbb{L}_{\beta f}\mathds{1}_{[\pm 1]} d\mu_\gamma
= 2\cosh(\beta\gamma) \int_X \mathds{1}_{[\pm 1]} d\mu_\gamma
= 2\cosh(\beta\gamma) \mu_{\gamma}([\pm 1]).
\end{align*}
Since
$
p=\mu_{\gamma}([+1])
= e^{\beta\gamma}/2\cosh(\beta\gamma)
$
and
$
1-p
=
\mu_{\gamma}([-1])
= e^{-\beta\gamma}/2\cosh(\beta\gamma)
$,
we finally get that $\gamma$ has to be a solution of the following equation
\begin{align*}
\gamma = 2p -1 = \frac{e^{\beta\gamma} -
e^{-\beta\gamma}}{2\cosh(\beta\gamma)}
= \tanh(\beta\gamma).
\end{align*}

The equation $\gamma = \tanh(\beta \gamma)$ has either one or
three solutions, depending on the value of $\beta$.
If $0<\beta\leqslant 1$ then $\gamma=0$ is the unique solution.
Otherwise, if $\beta>1$
then there is some $\gamma(\beta)\in (0,1)$ such that
$-\gamma(\beta),0$ and $\gamma(\beta)$ are all the solutions to the equation.

Now, we can justify the previous steps. Firstly, take $\gamma$ satisfying
$\gamma = \tanh(\beta \gamma)$. 
Secondly, note that in the previous computation of 
$\langle \mathds{1}, \mathbb{L}_{\beta\gamma} \mathds{1}_B \rangle_{\mu_\gamma}$
we can use the values we got for $\mu_{\gamma}([-1])$ and $\mu_{\gamma}([+1])$. 
Therefore, we have that  
\begin{equation*}
\begin{split}
\int_X \mathbb{L}_{\beta f}\mathds{1}_B d\mu_\gamma
& =
\frac{e^{-\beta\gamma}\mu_\gamma(B\cap[-1])}{\mu_\gamma([-1])}
+\frac{e^{\beta\gamma}\mu_\gamma(B\cap [+1])}{\mu_\gamma([+1])} \\
& =
\cosh(\beta\gamma)[\mu_\gamma(B\cap[-1]) + \mu_\gamma(B\cap[+1])] = \cosh(\beta\gamma)\mu_\gamma(B) \\
& = \cosh(\beta\gamma)\int_X \mathds{1}_B d\mu_\gamma.
\end{split}
\end{equation*}
Since the last identity holds for any measurable set $B$ it follows that 
$\langle \mathds{1},\mathbb{L}_{\beta f} \mathds{1}_B\rangle_{\mu_\gamma} = \langle \mathds{1}, \mathds{1}_B
\rangle_{\mu_\gamma}$ and so $\mu_\gamma$ is indeed a generalized conformal
measure associate with the potential $\beta f$. 
This reasoning actually proves the following proposition.

\begin{proposition}[Generalized Conformal Measures -- Curie-Weiss Model]
\label{prop-auto-medidas-CW}
Let $f$ be the potential defined given by \eqref{eq:cw_potencial} and for each $\gamma\in (-1,1)$
let $\mu_\gamma$ be a Bernoulli measure as defined above.  Then
\begin{enumerate}
\item $\mu_{\gamma}$ is a generalized conformal measure, if and only if, $\gamma$ is a solution of the equation
      $\gamma = \tanh(\beta\gamma)$;
\item for any solution $\gamma$ of the above equation,
      $2\cosh(\beta\gamma)$ is an eigenvalue of $\mathbb{L}^{*}_{\beta f}$.
\end{enumerate}
\end{proposition}

By Proposition \ref{prop-auto-medidas-CW} if $0<\beta<1$, then $\mu_{0}$
the symetric Bernoulli measure, with parameter $p=1/2$
is a generalized conformal measure associated to the eigenvalue $2$.
But on the other hand, if $\beta>1$, this measure still
is an eigenmeasure associated to the eigenvalue $2$,
but now there is two other Bernoulli measures $\mu_{\pm\gamma(\beta)}$
associate to a strictly bigger eigenvalue $2\cosh(\beta \gamma(\beta))$.

Now let us move the discussion to the eigenfunctions. We first observe that
for any fixed $\beta>0$, the operator $\mathbb{L}_{\beta f}:L^1(\mu_{0})\to L^1(\mu_{0})$, 
has the constant function as an eigenfunction
associated to the eigenvalue $\lambda=2$, that is,
$\mathbb{L}_{\beta f}\mathds{1}=2\mathds{1}$.

However, for $\beta>1$, which is above the critical point of the original Curie-Weiss model,
we can see more interesting phenomena, such as multidimensional eigenspaces.
Since $\beta$ is fixed, in what follows we will write
$\mu_\pm \equiv  \mu_{\pm\gamma(\beta)}$ to lighten the notation.

Consider the operator $\mathbb{L}_{\beta f}:L^1(\nu)\to L^1(\nu)$,
where $\nu\equiv \nu(t)\equiv t\mu_{+}+(t-1)\mu_{-}$
is a nontrivial convex combination of $\mu_{\pm}$.
The measurable sets
$B_+ = \{x\in X : m(x) = +\gamma(\beta)\}$
and $B_- = \{x\in X : m(x) = -\gamma(\beta)\}$ are chosen  in such way
they form a measurable partition of the space $X=B_+\cup B_-\cup N$ up
to a $\nu$-null set $N$. 
Note that $\mu_+(B_{+})=1$ and $\mu_-(B_{+})=0$, and so they behave similarly as the sets
constructed in
\ref{lema-conditional-coformal-reciprocal}.
Proceeding as before it is simple to argue that $\mathds{1}_{B_{\pm}}$ are two linearly independent
eigenfunctions for the adjoint of the the operator $\mathbb{L}_{\beta f}$, that is,
$\mathbb{L}_{\beta f}^*\mathds{1}_{B_{+}}=2\cosh(\beta \gamma(\beta))\mathds{1}_{B_{+}}$
and $\mathbb{L}_{\beta f}^*\mathds{1}_{B_{-}}=2\cosh(\beta \gamma(\beta))\mathds{1}_{B_{-}}$.

Regarding the operator
$\mathbb{L}_{\beta f}$ itself, it turns out that  the characteristic functions $\mathds{1}_{B_{\pm}}$
are also eigenfunctions, more precisely,
$\mathbb{L}_{\beta f}\mathds{1}_{B_+}=2\cosh(\beta\gamma(\beta))\mathds{1}_{B_+}$
and $\mathbb{L}_{\beta
f}\mathds{1}_{B_-}=2\cosh(\beta\gamma(\beta))\mathds{1}_{B_-}$.
To see this, remember that for any point $x\in B_+$,
$m(x)=\gamma(\beta)$, and so
\[
\mathbb{L}_{\beta f}\mathds{1}_{B_{+}}(x)
= \sum_{a \in \{-1,1\}} \mathds{1}_{B_{+}}(x)\exp (\beta a m(ax))
= (2\cosh {\beta \gamma}) \mathds{1}_{B_{+}}(x).
\]
The same is true for $B_{-}$, since $m(B_{-})=-\gamma(\beta)$.
Moreover, with a proper rewording, the proof of Theorem \ref{th:eigenspace-dimension} can
be adapted to this discontinuous case showing
that these are the only linear independent eigenfunctions
of $\mathbb{L}_{\beta f}$ (with the measure $\nu$).

We summarize this discussion with the following theorem.
\begin{theorem}\label{th:cw-spectre}
Let $\beta>0$, $f$ be the potential given by \eqref{def-transf-ope-CW-sections},
$\mu_{\gamma}$ the Bernoulli measure given by
\eqref{def-bernoulli-measure} and $\gamma(\beta),-\gamma(\beta)$ solutions of the
equation $\gamma = \tanh(\beta \gamma)$, then the following hold:
\begin{enumerate}
\item for $0<\beta \leqslant 1$:
      \begin{enumerate}
      \item the operator $\mathbb{L}_{\beta f}:L^1(\mu_0)\to L^1(\mu_0)$ has norm $\lVert
            \mathbb{L}_{\beta f} \rVert = 2$ and the symmetric
            Bernoulli measure $\mu_0$ is a generalized conformal measure,
            associated to ${\beta f}$, in the sense of Definition \ref{automedida}.

      \item the eigenspace of $\mathbb{L}_{\beta f}$, associated to eigenvalue $2$, 
      has\! dimension\! one and is spanned by $\mathds{1}$;
      \end{enumerate}
\item for $\beta > 1$:
      \begin{enumerate}
      \item  the operator $\mathbb{L}_{\beta f}:L^1(\nu)\to L^1(\nu)$, where
            $\nu=t\mu_{\gamma(\beta)}+(t-1)\mu_{-\gamma(\beta)}$ and $t \in (0,1)$, has operator norm
            $\lVert \mathbb{L}_{\beta f} \rVert = 2\cosh (\beta\gamma(\beta)) > 2$
            and $\nu$ is a generalized conformal measure of
            associated to $\beta f$, in the sense of Definition \ref{automedida}.

      \item for any non-trivial convex combination
            $\nu=t\mu_{\gamma(\beta)}+(t-1)\mu_{-\gamma(\beta)}$,
            the eigenspace of $\mathbb{L}_{\beta f}$ is
            two-dimensional and is spanned by $\{\mathds{1}_{B^+}, \mathds{1}_{B^-}\}$,
            where $B_{\pm} = \{x\in X : m(x) = \pm\gamma(\beta)\}$ and $m$ is given by \eqref{def-mag}.
      \end{enumerate}
\end{enumerate}
\end{theorem}

This gives us an example of a (discontinuous) potential
for which the eigenspace associated to its Ruelle operator
on $L^1(\nu)$ has dimension bigger than one. 
Note that for continuous potentials similar behavior could happen 
only if, we have phase transition,
in the sense of multiple extreme points in $\mathscr{G}^*$.
In some sense, this is the case in our example if we interpret $\mu_+$ and
$\mu_-$ as our extreme points.
The analogy with the results developed in Section \ref{sec:markov} 
can be extended a bit further, since $\mu_+$
and $\mu_-$ are fully supported and there are
disjoint full-measure the sets
$B_{+}$ and $B_{-}$ for each one of them,
as in Lemma \ref{lema-conditional-coformal-reciprocal}.

In what follows, we give another example where the dimension 
of the SHF for the extended operator associate with a continuous
potential has dimension strictly bigger than one.

\subsection{Multiple $g$-measures and the Space of Harmonic Functions}\label{sec-g-measures} 

In the context of one-sided shifts 
a positive and continuous function $g:X\to\mathbb{R}$ is called a $g$-function \cite{MR310193,MR412389}
if the potential $f=\log g$ is a normalized potential, i.e., $\mathscr{L}\mathds{1}=\mathds{1}$. 
In this case, the elements of $\mathscr{G}^{*}$ are called $g$-measures. 

There are several examples of $g$-functions for which the set of associated $g$-measures is not 
a singleton see, for example, the references \cite{bhs2017,MR1244665,MR3350377,MR0435352,MR0436850,FF70,MR3928619}.
In what follows, for the sake of concreteness, the reader can fix any of these examples.  

Let $g:X\to\mathbb{R}$ be a $g$-function such that the set of conformal measures $\mathscr{G}^{*}$
associated to $f=\log g$ satisfies $\mathrm{ex}(\mathscr{G}^{*})=\{\nu_1,\ldots,\nu_{n}\}$. 
Take $\nu$ as the barycenter of $\mathscr{G}^{*}$. It follows from the definition of a $g$-function that  $\mathbb{L}_{\nu_{j}}\mathds{1}=\mathds{1}$, for any $j=1,\ldots,n$.
And so the condition \textit{iii)} in the proof of 
Theorem \ref{th:eigenspace-dimension} is satisfied. Therefore, by arguing as in the theorem's proof, 
we can ensure the existence of 
measurable sets $B_1,\ldots, B_n$ such that $\nu_{i}(B_{j})=\delta_{K}(i,j)$ and $B_i\cap B_j = \emptyset$
whenever $i\neq j$. Moreover, $\{\mathds{1}_{B_1},\ldots, \mathds{1}_{B_n}\}$ is a basis for 
the SHF for $\mathbb{L}_{\nu}$, thus showing that the upper bound obtained in Theorem \ref{th:eigenspace-dimension} is saturated.

\subsection{On the existence of Harmonic Functions}\label{sec-suff-cond-eigenfunctions}

In this section we provide a set of sufficient conditions for the existence of a harmonic function
associated to $\rho(\mathscr{L})$ for the transfer operator or some of its extensions.
We begin with an abstract criterion  related to the equality of the operator norm and the spectral radius.
In the sequel we assume that $p:\mathscr{B}(E)\to [0,1]$ is a full support a priori probability measure. We
use the following notations for the operator norm
\[
\|\mathscr{L}\|_{\mathrm{op}}
\equiv
\sup_{\|\varphi\|_{\infty}=1} \|\mathscr{L}\varphi \|_{\infty}
\quad \text{and}\quad
\|\mathbb{L}\|_{\mathrm{op}}
\equiv
\sup_{\|\varphi\|_{L^1(\nu)}=1} \|\mathbb{L}\varphi \|_{L^1(\nu)}.
\]
Analogously, the spectral radius of $\mathbb{L}$ acting on $L^1(\nu)$
is denoted by $\rho(\mathbb{L})$.

\begin{proposition}\label{prop-opnorm-raio}
Let $f$ be a continuous potential and suppose that $\rho(\mathscr{L}) = \|\mathscr{L}\|_{\mathrm{op}}$
then the constant function $h\equiv 1$ is a harmonic function of $\mathscr{L}$.
\end{proposition}
\begin{figure}[h]
\centering
\includegraphics[width=6.0cm, height=3.9cm]{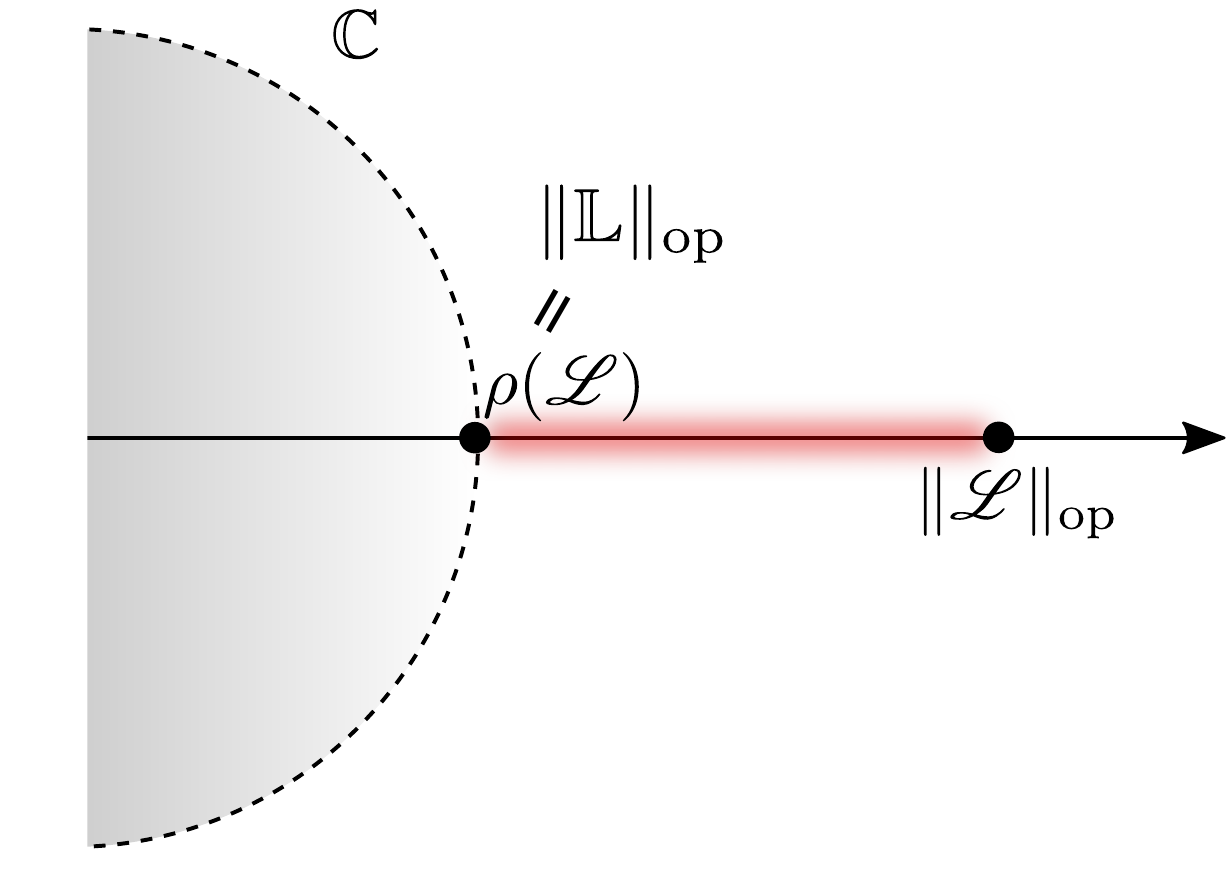}
\label{fig-prop-raio-norma}
\caption{In most cases, $\rho(\mathscr{L}) < \|\mathscr{L}\|_{\mathrm{op}}$ }
\end{figure}

\begin{proof}
From the positivity of $\mathscr{L}$
follows that $\|\mathscr{L}\|_{\mathrm{op}} = \|\mathscr{L}1\|_{\infty}$.
The map  $x\longmapsto	\|\mathscr{L}1\|_{\infty}-\mathscr{L}(1)(x)$
is continuous and takes only non-negative values. By integrating this function
and using the hypothesis and  above identity, we get
$
0
\leq
\int_{X} \|\mathscr{L}1\|_{\infty}-\mathscr{L}(1) \, d\nu
=
\|\mathscr{L}1\|_{\infty} - \rho(\mathscr{L})
=0.
$
Since integrand is continuous we conclude that $\|\mathscr{L}1\|_{\infty}=\mathscr{L}(1)(x)$
thus proving the proposition.
\end{proof}

Of course, there are several other sufficient conditions based on the regularity properties of
the potential $f$. One of the weakest regularity properties
is the famous bounded distortion condition.

Next, we present alternative necessary and
sufficient conditions that can be applied for
continuous  potentials that may not satisfy theses classical regularities properties.
They are derived from some
results in \cite{foguel} but for our setting.

From the discussions of Section \ref{sec:markov} it should be clear that the
problem of the existence of a harmonic function can be analyzed on each of the
supporting sets $B_k$ with $\nu_k(B_k)=1$ and
$B_j\cap B_k = \emptyset$. Therefore as in the proof of
Theorem \ref{th:eigenspace-dimension},
we will restrict ourselves, without loss of generality,
to extensions associated to extreme elements of $\mathscr{G}^*$.

\bigskip 

In the following,
we discuss Theorem E of \cite[p.45]{foguel} in the setting of the present paper.
This is an outstanding result in the theory of Markov processes
and roughly speaking it says that there exists  a measurable partition of the
space $X = A_0 \cup A_1$, called Hopf decomposition, for which
there is at least 
one harmonic function which is positive
on $A_1$, but there is no harmonic function
having a positive part in a subset of $A_0$ ($\nu$-a.e).

Recall that, whenever $\nu \in \textrm{ex}(\mathscr{G}^*)$,
Lemma \ref{lema-eigenfunctions-positivas} guarantees
that any non-negative harmonic function is
actually positive $\nu$-a.e..
In terms of the theory of Markov processes
this is the same as saying that  Hopf's decomposition is trivial in the sense that $X=A_1$.
On the other hand, when no harmonic function exists,
$X=A_0$, and so the decomposition is always trivial.

Note that these are precisely the hypotheses of
Corollaries 1 and 2 in \cite[p.45--46]{foguel},
which, translated back to our notation, read as follows

\begin{lemma}
Let $f$ be a continuous potential and $\nu\in\mathscr{G}^{*}$ an
extreme conformal measure. Then the following are equivalent:
\begin{enumerate}
\item there exists a unique, up to multiplication by a constant, $0 \not
      \equiv u \in L^1(\nu)$ such that $\mathbb{L}u = u$;
\item
      if $0 \leqslant v \in L^\infty(\nu)$, $v\not\equiv 0$, then $\liminf_n \langle v,
      \mathbb{L}^n1 \rangle_\nu > 0$;
\item
      if $0 \leqslant v \in L^\infty(\nu)$, $v\not\equiv 0$, then
      \begin{equation*}
      \liminf_N \frac{1}{N}  \sum_{n=0}^{N-1} \langle v, \mathbb{L}^n1 \rangle_\nu > 0;
      \end{equation*}
\item
      There is no set $A$ of positive $\nu$-measure for which
      \begin{equation*}
      \lim_N \frac{1}{N} \sum_{n=0}^{N-1} (\mathbb{L}^*)^n\mathds{1}_A = 0 \quad \mathrm{(}
      \nu-\mathrm{a.e.)};
      \end{equation*}
\item
      There is no set $A$ of positive measure for which
      \begin{equation*}
      \lim_N \frac{1}{N} \sum_{n=0}^{N-1} (\mathbb{L}^*)^n\mathds{1}_A = 0 \text{ uniformly ($\nu$-a.e.)}.
      \end{equation*}

\end{enumerate}
\end{lemma}

\subsection{Phase Transitions in Lattice Spin Systems}\label{sec-phase-transition}

Let $E$ be a standard Borel space, $X=E^{\mathbb{N}}$ and 
$\mathscr{B}(X)$ the product $\sigma$-algebra on $X$.
The space $X$ is regarded as a metric space  with the distance
$d_{X}(x,y)=\sum_{n=1}^{\infty} 2^{-n}\min\{d_{E}(x_n,y_n),1\}$.
For each $i\in \mathbb{N}$, let $\pi_{i}:X\to E$ be the standard projections 
as defined in Subsection \ref{sec-app-EqSt}. 
For each  $\Lambda \Subset \mathbb{N} $ (finite subset) we consider the following sub-sigma-algebras 
$\mathscr{B}_{\Lambda} \equiv \sigma(\pi_{j}: j\in \Lambda )$ and
$\mathscr{T}_{\Lambda} \equiv \sigma\big(\cup_{\Gamma} \mathscr{B}_{\Gamma} : \Gamma\Subset \Lambda^c  \big)$.

A probability kernel $\gamma_{\Lambda}$ is called a proper probability kernel from $\mathscr{T}_{\Lambda}$ to $\mathscr{B}_{\Lambda}$ if 
\begin{itemize}
	\item[\textit{i)}]  $\gamma_{\Lambda}(\cdot|x)$ is probability measure on $(X,\mathscr{B}(X))$ for any $x\in X$;
	\item[\textit{ii)}]  $\gamma_{\Lambda}(A|\cdot)$ is $\mathscr{T}_{\Lambda}$-measurable for any $A\in\mathscr{B}(X)$
	\item[\textit{iii)}]  $\gamma_{\Lambda}(A\cap B|x)=1_{B}(x)\gamma_{\Lambda}(A|x)$ 
	for any $A\in\mathscr{B}(X)$, $B\in \mathscr{T}_{\Lambda}$ and $x\in X$.
\end{itemize}
The family $\gamma\equiv (\gamma_{\Lambda})_{\Lambda\Subset \mathbb{N}}$ is 
	said to be consistent if 
	\[
	\int_{X} \gamma_{\Lambda}(A|x)\, d\gamma_{\Gamma}(\cdot|x) = \gamma_{\Gamma}(A,x),
	\ \text{whenever}\ \emptyset\subsetneq \Lambda\subset \Gamma.
	\] 

A \textit{specification} with parameter  set $\mathbb{N}$ and state space $E$ is a family 
$\gamma\equiv (\gamma_{\Lambda})_{\Lambda\Subset \mathbb{N}}$  such that 
$\gamma_{\Lambda}$ is a proper probability kernel from $\mathscr{T}_{\Lambda}$ to $\mathscr{B}_{\Lambda}$
and $(\gamma_{\Lambda})_{\Lambda \Subset \mathbb{N} }$ 
satisfies the consistency condition.

Let  $\gamma$ be a specification with parameter set $\mathbb{N}$ and state space $E$.
The set of all Borel probability measures defined by 
\[
	\mathscr{G}^{\text{DLR}}(\gamma)
	=
	\left\{
	\mu\in \mathscr{M}_{1}(X): \mu(A|\mathscr{T}_{\Lambda})(x) = \gamma_{\Lambda}(A,x) \ \ \mu-a.s.
	\right\}
\] 
will be called the set of DLR-Gibbs Measures determined by $\gamma$.
In this context of Statistical Mechanics, if $\#\mathscr{G}^{\textrm{DLR}}(\gamma)>1$, then we say that we have phase transition.
	
\begin{theorem}[Georgii, \cite{MR2807681}]\label{teo-cofinal}
Let $\gamma = (\gamma_{\Lambda})_{\Lambda \Subset \mathbb{N} }$ be a specification with parameter set $\mathbb{N}$ and state space $E$.
	Then the following statements are equivalent:
	\begin{enumerate}
		\item $\mu\in \mathscr{G}^{\text{DLR}}(\gamma)$;
		\item $\mu(A) = \int_{X} \gamma_{\Lambda}(A,x)\, d\mu(x) \equiv  \mu\gamma_{\Lambda}(A)$,\quad  for all $A\in\mathscr{F}$ and $\Lambda \Subset \mathbb{N}$;
		\item  There is a cofinal collection $\{\Gamma_{\alpha}: \Gamma_{\alpha}\Subset \mathbb{N}, \forall \alpha\in I \}$
		(i.e., directed by inclusion and, for all $\Lambda \Subset \mathbb{N}$, there is an index $\alpha \in I$ such that 
		$\Lambda \Subset \Gamma_{\alpha}$ ) satisfying $\mu(A) = \mu\gamma_{\Lambda}(A)$, for all $A\in \mathscr{F}$.
	\end{enumerate}
\end{theorem}

\begin{example} Let $E=\{0,1\}$ be a finite alphabet, $f\in C(X)$ 
a continuous potential and $\mathscr{L}:C(X)\to C(X)$ the transfer operator defined by
\[\mathscr{L}(\varphi)(x)
\equiv
\sum_{y\in\sigma^{-1}(x)} \exp(f(y))\varphi(y)
\] 
For each $n\in\mathbb{N}$, consider the volume $\Lambda_n = \{1,2,\ldots,n\}$ and the kernel 
\begin{equation}\label{eq-specification-example}
	\gamma_{\Lambda_n}(A|x)\equiv \displaystyle
	\frac{\mathscr{L}^{n}(1_{A})(\sigma^{n} x)}{\mathscr{L}^{n}(1)(\sigma^{n}x) }.
\end{equation}
In \cite{CLS20} the authors show that the family of kernels $(\gamma_{\Lambda_n})_{n\in\mathbb{N}}$ 
can be naturally extended to a specification $\gamma = (\gamma_{\Lambda})_{\Lambda \Subset \mathbb{N}}$.
By Theorem \ref{teo-cofinal} we have that $\mathscr{G}^{\mathrm{DLR}}(\gamma)$ 
does not depend on the choice of this extension, 
since the collection $(\Lambda_n)_{n\in\mathbb{N}}$ is a cofinal collection. 
Moreover, for any continuous potential $f$, the equality  
$\mathscr{G}^{\mathrm{DLR}}(\gamma)= \mathscr{G}^{*}$ holds.

Therefore, if for some $\nu\in \mathscr{G}^{*}$
the SHF for $\mathbb{L}: L^1(\nu)\to L^1(\nu)$ has dimension bigger than one, then 
as a consequence of Theorem \ref{teo-unica}, $\nu$ can not be an extreme measure in $\mathscr{G}^*$.
By Krein-Milman Theorem it follows that $\#\mathscr{G}^{*}>1$ and so  
$\#\mathscr{G}^{\mathrm{DLR}}(\gamma)>1$, that is, 
we have a first-order phase transition in the sense of Dobrushin.

In practice, we will begin with a continuous potential $f$ and them construct the specification 
$(\gamma_{\Lambda})_{\Lambda \Subset \mathbb{N} }$, where
$\gamma_{\Lambda_n}(A|x)$ is given by expression \eqref{eq-specification-example}.
Next, we fix an arbitrary point $x\in X$. Usually, this point is chosen within the set 
where the potential attains its maximum value (recall that we already applied this strategy in Subsection
\ref{sec-Curie-Weiss} since the plus and minus condition satisfies this property). 
Since $X$ is compact 
we can ensure that the sequence of probability measures 
$(\gamma_{\Lambda_n}(\cdot|x))_{n\in\mathbb{N}}$
has at least one cluster point $\nu$ which is, as mentioned before, an element 
in $\mathscr{G}^{*}$. With our conformal measure in hands,  
we construct our extension $\mathbb{L}:L^1(\nu)\to L^1(\nu)$. 
Once we have this extension we look at its Perron-Frobenius eigenvector space.
Finally, if this space has dimension bigger than one we can use the theory developed here
to ensure that the system has phase transition. 
In most practical cases, to bound from below the cardinality of $\mathscr{G}^{*}$ or 
the dimension of the subspace $\mathscr{H}\equiv \{h\in L^1(\nu): \mathbb{L}h=h \}$ 
will be extremely hard problems. 
When the alphabet $E$ is finite, the problem of lower bounding 
the cardinality of $\mathscr{G}^{*}$ is expected to be simpler than 
the one of lower bounding the dimension of $\{h\in L^1(\nu): \mathbb{L}h=h \}$,
because it can be reduced to the entropy-energy balance problem. However, as 
far as uncountable alphabets are concerned, we think that 
both methods can complement each other. 
For example, in some cases such as 
product type potentials in infinite alphabets
the second problem can be solved, see \cite{JLMS:JLMS12031}. 
They probably are toy models that could be used to expand the applications
of our theory.   

The sufficient condition for phase transition, presented in this section, 
motivates the following interesting question.
Fix a continuous potential $f$, and take $\nu$ as the barycenter
of $\mathscr{G}^{*}$. 
Let $\mathscr{H}$ denote the SHF 
for $\mathbb{L}:L^1(\nu)\to L^1(\nu)$. 

Roughly speaking, in most cases where $\mathrm{dim}\, \mathscr{H}=1$,
we normally do not have phase transition. On the other hand, if 
$\mathrm{dim}\, \mathscr{H}>1$ we do have phase transition. This naturally 
leads us to ask what happens when we do not have harmonic functions,
that is, is there any interesting physical phenomenon occurring when 
$\mathrm{dim}\, \mathscr{H}=0$? 
At this moment, we do not know the answer, but we would like to remark that 
this property is verified for some models used to study the intermittence phenomenon, 
such as the Manneville-Pomeau model. 



		
\end{example}

\subsection{Functional Central Limit Theorem}
\label{FCLT}
In this section we study the validity of a Functional Central Limit Theorem 
in our setting. We use the version obtained here
to prove a new result in Statistical Mechanics.

It is well-known that there are several techniques to
reduce the problem of proving a FCLT for a Markov process to
the problem of verifing some analytical condition on the
associated transfer operator. See for instance \cite{BataCLT,Gordin,MR1862393,Wood}.
A classical one is to find  a solution of Poisson's equation,
which is the approach we will follow here. 

Let us recall the definitions of the transfer operator and Poisson's equation.
Consider a homogeneous Markov chain $(Z_{n})_{n\in\mathbb{N}}$
with state space $S$ and transition probability $p(\cdot,\cdot)$.
Although a similar notation is adopted, we shall remark that the transition probability has no relation with
the a priori measure considered in this paper.
The transfer operator $P$ induced by the transition probability $p$
is defined as follows: given a non-negative measurable real-valued function $\varphi$,
the action of $P$ on $\varphi$ is a non-negative measurable function $P\varphi$ given by
\begin{equation} \label{transferoperator}
P\varphi(x)=\int \varphi(y)\,p(x,dy).
\end{equation}
For a measurable function $\varphi$ not necessarily non-negative,
we write $\varphi=\varphi^{+}-\varphi^{-}$ as a difference of non-negative
ones and define
$P\varphi(x)=P\varphi^{+}(x)-P\varphi^{-}(x)$,
if $P\varphi^{+}(x)$ and $P\varphi^{-}(x)$ are both finite.
A probability measure $\mu$ on $S$ is said to be  a {\it stationary measure} if
\[
\mu(A)=\int p(x,A)\, d\mu(x)
\]
for every Borel set $A\subset S$.

In \cite[p.1340]{BhLe} a simple condition on the transfer
operator $P$ guaranteeing the validity of a FCLT is formulated.
Namely, assume that $(Y_{n})_{n\in\mathbb{N}}$ is an ergodic stationary Markov
chain whose stationary distribution is $\mu$.
Note that $P$ takes $L^{2}(\mu)$ into $L^{2}(\mu)$.
Given an almost surely non-constant observable
$\phi\colon S\to \mathbb{R}$ such that  $\phi\in L^{2}(\mu)$ and $\int_{X} \phi\, d\mu=0$,
consider \emph{Poisson's equation}
\[
(I-P)\upsilon=\phi.
\]
Assume that there is a solution $\upsilon\in L^{2}(\mu)$ and  let
$\varrho\equiv \mu(\upsilon^2)-\mu(P\upsilon^2)>0$ (ergodicity implies $\varrho>0$).
Consider the stochastic process $Y_{n}(t)$ given by
\[
Y_{n}(t)=\displaystyle\frac{1}{\varrho\sqrt{n}}\sum_{j=0}^{[nt]}\phi(Z_{j}), \qquad 0\leqslant t<\infty
\]
taking values in the space \( D[0,\infty) \) of real-valued right continuous
functions on \( [0,\infty) \) having left limits, endowed with the Skorohod topology.
Then, the process $ Y_{n}(t)$ converges in distribution
(weak-$*$ convergence) to the Wiener measure on \( D[0,\infty) \).

The FCLT stated above is proved by reducing the problem to the martingale case.  See  \cite{BataCLT,Gordin} for this reduction 
and Billingsley \cite[Theorem 18.3]{MR1700749} for a FCLT for martingale differences.


In our case this theory is applied as follows.
Let $f$ be a normalized potential, i.e., $\mathscr{L}\mathds{1}=\mathds{1}$.
For each measurable set $A$, define the transition
probabilities by
\[
p(x,A)=\mathbb{L}(1_A)(x),
\]
where the operator $\mathbb{L}$ is the extension constructed  in Proposition \ref{prop-extensao-l1}.
One can easily show that the above expression defines a transition
probability  kernel of some Markov chain $(Z_n)_{n\in\mathbb{N}}$ taking values in $X$.
A straightforward computation shows that the induced transfer operator satisfies 
$P\varphi (x)=\mathbb{L}(\varphi)(x)$. Therefore, any probability measure fixed by
the dual operator  $\mathscr{L}^*$ is a stationary measure for $P$.

In the sequel, as usual, we denote by $E_{\mu}$ the expectation with respect to the joint law of the  Markov
Chain $(Z_n)_{n\in\mathbb{N}}$ with stationary measure
$\mu$.

The distributional relation between the Markov chain $(Z_n)_{n\in\mathbb{N}}$
and the underlying dynamics of the operator $\mathbb{L}$ is given
by the following lemma, whose proof can be found in reference \cite[p.85]{MR1862393}.
\begin{lemma}\label{lema-eq-distribution}
Let $(Z_n)_{n\in\mathbb{N}}$ be the Markov chain defined above,
$n\geqslant 1$, $g:X^{n}\to \mathbb{R}$ a positive measurable function, 
and $\mu$ a stationary measure.
Then we have
\[
\int_X g(x,\sigma(x), \ldots, \sigma^{n-1}(x))\,d\mu(x)=E_\mu[g(Z_n, Z_{n-1},\ldots,Z_1 )].
\]
\end{lemma}

Here we focus on functions $g$ of the form $g =\mathds{1}_{A}\circ h$, where
$A=(-\infty,t]$ is some suitable interval on the real line, $h:X^n\to \mathbb{R}$ is given by
$h(z_1,\ldots, z_n) = \phi(z_1)+\ldots +\phi(z_n)$, with $\phi:X\to\mathbb{R}$ being a positive
function in some Banach space, for example, the space of  H\"older continuous functions.

In these particular cases, by taking $A=(-\infty,t\sqrt{n}]$, and $h$ as above, we get from
Lemma \ref{lema-eq-distribution} the following identity
\begin{align*}
\mu(\{\textstyle\sum_{j=0}^{n-1}\phi\circ\sigma^{j}(x)\leqslant t\sqrt{n} \})
 & =
\int_X \mathds{1}_{(-\infty,t\sqrt{n}]}\circ h(x,\sigma(x), \ldots, \sigma^{n-1}(x))\,d\mu(x)
\\[0.3cm]
 & =E_\mu[\mathds{1}_{(-\infty,t\sqrt{n}]}\circ h(Z_n, Z_{n-1},\ldots,Z_1 )]
\\[0.3cm]
 & =
P_{\mu}(\{ \textstyle\sum_{j=1}^n\phi(Z_j)\leqslant t\sqrt{n}\}).
\end{align*}

This relation implies, for example, that if random variables
$(\phi(Z_n))_{n\in\mathbb{N}}$ are distributed according to
$P_{\mu}$ and obey a FCLT, then
the random variables  $(\phi\circ\sigma^n)_{n\in\mathbb{N}}$
distributed according to $\mu$ also obey a FCLT. 
Here the stationary measure $\mu$ is a fixed point of $\mathscr{L}^{*}$. 
Since we are assuming $\mathscr{L}\mathds{1}=\mathds{1}$, in many cases
such  measure is an \textit{equilibrium state} associated to the potential $f$,
see \cite{MR3897924} for more details on equilibrium states on uncountable alphabet settings.

The previous observations gain in relevance  when we further particularize the above setting
by taking $E=[a,b]$ (bounded and closed interval on the real line) and $\phi:X\to \mathbb{R}$
as being the projection on the first coordinate. In this case the sequence $(\phi\circ\sigma^n)_{n\geqslant
0}$
can be regarded as  the standard coordinate process 
on $(X,\mathscr{B}(X),\mu)$. This coordinate processes
exhibits what we call a Gibbsian dependence. Such dependence, in general,
is stronger than a Markovian dependence, meaning that the stochastic process defined by
$(\phi\circ\sigma^n)_{n\geqslant 0}$ may have an infinite memory,
in the usual sense of conditioning, as in Markov chains.  On the other hand,
Lemma \ref{lema-eq-distribution} says that, by considering a bigger space, this coordinate process
exhibiting a Gibbsian dependence can actually be described, in law, by a suitable Markov process.

\begin{theorem}\label{functionalclt}
Let $P$ be the transfer operator induced by the extension $\mathbb{L}$
associated with a continuous
and normalized potential and $\mu\in \mathscr{G}^{*}$.
Let $\phi:X\to \mathbb{R}$ be a non-constant observable  in $L^{2}(\mu)$
satisfying $\mu(\phi)=0$. If
there exists a solution $\upsilon\in L^{2}(\mu)$ for Poisson's equation $(I-\mathbb{L})\upsilon=\phi$,
then the stochastic process $Y_{n}(t)$, given by
\begin{align}\label{eq-Ynt}
Y_{n}(t)=\displaystyle\frac{1}{\varrho\sqrt{n}}\sum_{j=0}^{[nt]}\phi\circ \sigma^{j}, \qquad 0\leqslant t<\infty,
\end{align}
where $\varrho=\mu(\upsilon^2)-\mu(P\upsilon^2)$, converges in distribution to the Wiener measure in
$D[0,\infty)$.
\end{theorem}

\begin{example}[Spectral Gap]\label{ex-spe-gap}
In this example, we show how to apply the above theorem to a transfer operator whose  action on the space of
H\"older continuous functions have the spectral gap property. This is very well-known when $E$ is a finite set and
have been proved by several different methods. Here the aim is to present a similar result in the setting of compact
metric alphabets.

Let $X=E^{\mathbb{N}}$, where $E$ is a compact metric space,
and $f:X\to\mathbb{R}$ be an $\alpha$-H\"older potential, that is,
$f$ is a potential satisfying
\[
\mathrm{Hol}_{\alpha}(f) \equiv \sup_{x\neq y} \frac{|f(x)-f(y)|}{d(x,y)^{\alpha}}<\infty.
\]
Up to summing a coboundary term, see \cite{MR4026981,MR3377291},
we can assume that $f$ is a normalized potential, meaning that $\mathscr{L}\mathds{1}=\mathds{1}$.
In this setting, the authors of \cite{MR3377291} showed that $\mathscr{G}^{*}=\{\mu\}$ is a singleton, and its
unique measure is $\sigma$-invariant. In \cite{MR4026981} the authors showed that the transfer operator acts on
the space of $\alpha$-H\"older functions with a spectral gap.

Let  $\phi$ be an arbitrary $\alpha$-H\"older observable for which $\mu(\phi)=0$.
We will show that in this case, we can always get a solution for Poisson's equation  $(I-P)\upsilon=\phi$.
Indeed, following the notation of reference \cite{MR4026981}, and taking $\psi\equiv 1$ and $\varphi=\phi$
in Theorem 3.1 we get the existence of two constants $0<s<1$ and $C>0$ such that
\[
\|\mathbb{L}^n \phi\|_\infty\leqslant C s^n,
\]
and as a consequence $\|\mathbb{L}^n \phi\|_2 \leqslant C s^n$, which  implies that
$\upsilon=-\sum_{n=0}^{\infty}\mathbb{L}^n\phi$
is a well-defined element of $L^2(\mu)$ and also a solution for Poisson's equation.
Therefore the stochastic process $Y_n(t)$ as defined in \eqref{eq-Ynt}
converges in distribution to the Wiener measure in $D[0,\infty)$.
\end{example}

The aim of the next example is to show the validity of the FCLT
for a large class of observables
in a situation where we do not have the spectral gap property,
by using Poisson's equation.
The argument is much more involved and it is based on the previous theorem
together with a series of recent results \cite{MR4059795,MR3538412}
about the maximal spectral data of the Ruelle operator.

\begin{example}[Absence of Spectral Gap]\label{ex-no-spe-gap}
In this example we consider a Dyson type model for ferromagnetism on the one-sided lattice.
Before presenting this model, we need to introduce some notation.

We start by remembering that a \emph{modulus of continuity} is a
continuous, increasing and concave function
$\omega:[0,\infty)\to [0,\infty)$ such that $\omega(0)=0$.
We say that $f:X\to \mathbb{R}$ is
$\omega$-H\"older continuous if there is a constant $C>0$ such that
\[
|f(x)-f(y)|\leqslant C \omega(d(x,y))
\]
for any $x,y\in X$. The smallest constant $C$ satisfying the above inequality
will be denoted by $\text{Hol}_\omega(f)$.
Denote by $C^\omega(X)$ the space of all such functions. If we set
$\|f\|_\omega=\|f\|_\infty+\text{Hol}_\omega(f)$, it
is a straightforward calculation to see that
$(C^\omega(X), \|\cdot \|_\omega)$ is a Banach algebra.

Note that, for $\alpha,\beta\geqslant 0$, the function
$\omega_{\alpha+\beta \log}(r) \equiv r^\alpha\log(r_0/r)^{-\beta}$
defines a modulus of continuity. In this particular case,
we denote the space of $\omega_{\alpha+\beta\log}$-H\"older continuous functions by
$C^{\alpha+\beta \log}(X)$. In particular, if $\alpha=0$,
which is the case we are interested here, we simply write $C^{\beta \log}(X)$.

Let $X=\{-1,1\}^{\mathbb{N}}$, endowed with the metric $d(x,y) = 2^{-N(x,y)}$, where the number
$N(x,y) \equiv \inf\{i\in\mathbb{N} : x_j=y_j, 1\leqslant j\leqslant i-1\  \text{and} \ x_i\neq y_i  \}$  .
The \emph{Dyson} potential on the one-sided lattice is given by the following expression
\[
f(x)=\sum_{n=2}^\infty \dfrac{x_1x_n}{n^{2+\varepsilon}}.
\]
One can easily show that $f$ is not a H\"older continuous function with respect to $d(x,y)$.
In addition, the transfer operator $\mathscr{L}_{f}$, associated with a Dyson
potential $f$,
does not leave the space of H\"older continuous functions invariant.
But, on the other hand, it leaves invariant a bigger subspace of $C(X)$, called
the Walters space.
Although there exists $\bar{f}$ in the Walters space cohomologous to $f$, neither
$\mathscr{L}_{f}$ nor $\mathscr{L}_{\bar{f}}$ acts with the spectral gap property on this
subspace, see \cite{MR3538412}.
Therefore, the techniques employed in the previous example can not be used here to
obtain a FCLT for the unique \textit{equilibrium state}
$\mu_{f}\in \mathscr{G}^{*}(\bar{f})$.
Actually, for this potential, it is not clear how to find a proper subspace of $C(X)$ where
Hennion and Hervé Theory (Theorems A, B and C \cite[p.12]{MR1862393})
could be applied.

The aim of this example is to prove that the stochastic process
\[
Y_{n}(t) = \displaystyle\frac{1}{\varrho\sqrt{n}}\sum_{j=0}^{[nt]}\phi\circ \sigma^{j}, \qquad 0\leqslant t<\infty,
\]
where $\varrho=\mu_{f}(\upsilon^2)-\mu_{f}(P\upsilon^2)$,
converges in distribution to the Wiener measure in
$D[0,\infty)$ for any observable $\phi\in C^{\epsilon\log}$, and $\epsilon>2$.

\medskip
Firstly, we show  that $f\in C^{\varepsilon \log}(X)$. Indeed,
\begin{align*}
|f(x)-f(y)| & \leqslant  \sum_{n=N(x,y)+1}^{\infty} \frac{2}{n^{2+\varepsilon}}
\\
            & \leqslant\frac{2}{(N(x,y)+1)^{2+\varepsilon}}\left( 1+ \int_{1}^{\infty}\left(\frac{N(x,y)+1}{N(x,y)+t+1}
\right)^{2+\varepsilon} \,dt \right)
\\[0.2cm]
            & =\frac{2}{(N(x,y)+1)^{2+\varepsilon}}\left( 1+
\frac{(N(x,y)+1)^{2+\varepsilon}}{1+\varepsilon}\frac{1}{(N(x,y)+2)^{1+\varepsilon})} \right)
\\[0.1cm]
            & \leqslant \frac{2}{N(x,y)^{2+\varepsilon}}\left( N(x,y)+ 2^3\frac{N(x,y)}{1+\varepsilon} \right)
\\[0.1cm]
            & \leqslant \frac{2\times 10N(x,y)}{N(x,y)^{1+\varepsilon}}=\frac{20}{N(x,y)^{\varepsilon}}
\leqslant\frac{20}{[\log(2^{N(x,y)})]^{\varepsilon}}
\\[0.1cm]
            & =20 \log(2^{N(x,y)})^{-\varepsilon}= \log(r_02^{N(x,y)})^{-\varepsilon}=\omega(2^{-N(x,y)})
\\[0.1cm]
            & =\omega\circ d(x,y),
\end{align*}
where $\omega(r)=\log(r_0/r)^{-\varepsilon}$.

The previous estimate holds for any $\varepsilon>0$, but  to solve Poisson's equation later, we will
need to restrict ourselves to $\varepsilon>2$.
Observe that, from the definition of $d$, we have $d(a_1\cdots a_j x, a_1\cdots a_j y)=2^{-j}d(x,y)$.
By using this identity and the previous inequality, we conclude, for $0\leqslant j\leqslant n-1$, that
\[
|f(\sigma^j(a_1\cdots a_nx)) - f(\sigma^j(a_1\cdots a_ny))|
\leqslant \omega (2^{n-j}d(x,y)).
\]
Therefore,
$
|\sum_{j=0}^{n-1}f(\sigma^j(a_1\cdots a_nx))-f(\sigma^j(a_1\cdots a_ny))|\leqslant
\sum_{j=0}^{n-1}\omega (2^{-j}d(x,y)).
$
We recall that the summands on the rhs can be written as
\[
\omega(2^{-j}d(x,y))
=\Big[\log\big(\dfrac{r_0}{2^{-j}d(x,y)}\big)\Big]^{-\varepsilon}
=\Big[ j\log 2+\underbrace{\log\left(\dfrac{r_0}{d(x,y)}\right)}_{\equiv A}\Big]^{-\varepsilon}.
\]
Thus,
\begin{align*}
\sum_{j=0}^{n-1}\omega (2^{-j}d(x,y))
 & \leqslant \sum_{j=0}^{\infty}\omega
(2^{-j}d(x,y))=\sum_{j=0}^{\infty}\dfrac{1}{\left[j\log 2+A\right]^\varepsilon}
\\[0.3cm]
 & \leqslant  \lim_{\delta\to 0}\displaystyle\int_{\delta}^\infty\dfrac{dx}{[(\log 2)\cdot	x+A]^\varepsilon}
=\lim_{\delta\to 0}\dfrac{\log 2}{\varepsilon+1}\dfrac{1}{[(\log 2)\cdot  x+A]^{\varepsilon
-1}}\Big|_{\delta}^{\infty}
\\[0.3cm]
 & =
\underbrace{\dfrac{\log 2}{\varepsilon+1}}_{\equiv C}\dfrac{1}{\left[\log\left(\dfrac{r_0}{d(x,y)}\right)\right]^{\varepsilon-1} }.
\end{align*}
Showing that
\[
\sum_{j=0}^{n-1}\omega (2^{-j}d(x,y))\leqslant C \tilde{\omega}(d(x,y)),
\]
where $\tilde{\omega}(r)=\log(r_0/r)^{-(\varepsilon-1)}$. This
proves the existence of a constant $C>0$ such that, for every $n\in \mathbb{N}$,
\begin{equation}\label{flat}
\big|\sum_{j=0}^{n-1}f(\sigma^j(a_1\cdots a_nx))-f(\sigma^j(a_1\cdots a_ny))\big|\leqslant C\tilde{\omega} (d(x,y)).
\end{equation}
As a consequence of the above inequality we get two things, firstly $f\in
C^{(\varepsilon-1) \log}(X)$, and secondly,  from  \eqref{flat}, it  is easy to
see that the Dyson potential is \emph{a flat potential } with
respect to the \emph{natural coupling}, and $\tilde{\omega}$,  see
reference \cite{MR4059795} Definitions 2.1.3	and 5.2.

The previous discussion allows us to apply Theorems 3.2 and 4.1
of \cite{MR4059795} obtaining a strictly positive  eigenfunction
$h\in C^{(\varepsilon-1) \log}(X)$ associated with the eigenvalue $\rho(\mathscr{L})>0$.
By using this eigenfunction,  we can construct a normalized potential $\bar{f}\in C^{(\varepsilon-1) \log}(X)$
cohomologous to $f$ given by ${\bar f}=f+\log h-\log h\circ \sigma-\log \rho(\mathscr{L}_{f})$.
Moreover, a simple computation shows that $\bar{f}$ is also  a flat potential.
Recall that $\mathscr{L}_{\bar{f}}\mathds{1}=\mathds{1}$ and $\mathscr{L}_{\bar{f}}^{*}\mu_{f}=\mu_{f}$.

Therefore, we are in conditions to apply Theorem 5.8 of \cite[p.31]{MR4059795}
to prove the existence of a constant $D>0$ such that, for any
$\varphi\in C^{(\varepsilon-1)\log }(X)$ satisfying $\int_{X} \varphi\, d\mu_f=0$, we have that
\[
\|\mathscr{L}_{\bar{f}}^n\varphi\|_\infty \leqslant \dfrac{D}{n^{\varepsilon-1}}.
\]
Since we are assuming $\varepsilon>2$,
we get $\sum_{n=2}^\infty\|\mathscr{L}_{\bar{f}}^n\varphi\|_\infty\leqslant
\sum_{n=2}^\infty C n^{-(\varepsilon-1)}<\infty,$ which implies that
$v=-\sum_{n=0}^\infty\mathscr{L}_{\bar{f}}\varphi$ is a well defined element of $L^2(\mu_f)$
and also a solution for Poisson's equation which allows us to apply Theorem
\ref{functionalclt}, concluding the demonstration of the validity of a FCLT for this example.
\end{example}

We believe that the above example can be generalized to long-range $O(N)$ models
on the $\mathbb{N}$ lattice.  In the $O(N)$, for $N\geqslant 2$, the fibers are uncountable and given by
$E= \mathbb{S}^{n-1}$, the unit sphere in the $n$-dimensional Euclidean space.
The potential is similar and  given by the expression
$
f(x) = \sum_{n=1}^{\infty} J(n)\langle x_1,x_{n+1}\rangle,
$
where $J(n)=O(n^{-2-\varepsilon})$ and the scalar multiplication 
is replaced by the usual inner product of $\mathbb{R}^n$.
Most of the arguments above are easily generalized to this context, but the last step,
which is the inequality $\|\mathscr{L}_{\bar{f}}^n\varphi\|_\infty \leqslant D n^{1-\varepsilon}$,
would require a non-trivial extension of the main result of \cite{MR4059795}
to uncountable alphabets, which, as far as we know, remains an open problem.

We also remark that the FCLT obtained in Example \ref{ex-no-spe-gap} can be alternatively
obtained by combining the CLT in \cite{MR0503333} with tightness, proved in 
\cite{MR624694}. However, in this way, the interactions have to be so that the Gibbs measures
associated with them  satisfy the FKG inequality. 
Moreover, the observables have to be local functions,
the space $E$ has to be a subset of the real line, and the a priori measure needs to be
symmetric, that is, invariant by the map $T:E\to E$ given by $T(x)=-x$.
All of theses hypothesis are not required by our techniques. 
On the other hand, the combination of \cite{MR0503333} and \cite{MR624694} allows the exponent in the interaction 
to be smaller than two, provided the inverse temperature is small enough. 
Nevertheless, our method also works for non-ferromagnetic interactions
for which the FKG inequality might not hold.

\section{Concluding Remarks}\label{sec-concluding-rmk}

It would be very interesting to extend the results obtained in 
Subsection \ref{sec-g-measures} to a larger class of potentials
not necessarily defined by $g$-functions. The aim would be to
find a class of potentials for which the number of extreme conformal
measures is equivalent to the dimension of the SHF of $\mathbb{L}$.
This would provide other examples of continuous potentials where
the upper bound established on Theorem \ref{th:eigenspace-dimension} is
saturated.

If this equivalence holds on the uniformly absolutely summable (UAS)
class of potentials, the definition of phase transition
in the original setting devised by Dobrushin would be equivalent
to a multidimensional SHF. If this is so, a change in the
dimension of the SHF of $\mathbb{L}$
with the variation of the parameter
$\beta$, for a continuous potential $\beta f$, would
be an alternative definition of phase transition out of the UAS
class. This definition has the advantage to detect
the transition from $\dim{\mathscr{H}=0}$ to $\dim{\mathscr{H}>0}$
not present in the original definition, as we mentioned
at the end of section \ref{sec-phase-transition}.

Particular examples of continuous potentials
in the UAS class for which it would be interesting
to detect a multidimensional SHF of $\mathbb{L}$
(at low temperatures)
are the ones defining the Dyson model, 
given by
\[
f(x) = \sum_{n=1}^{\infty} \frac{x_1x_{n+1}}{n^{1+\varepsilon}}, \qquad 0<\varepsilon <1.
\]

A crucial step in this direction  has been proved recently by
Johansson, \"Oberg and Pollicott in \cite{MR3928619}. They have
shown that the set of conformal measures $\mathscr{G}^{*}(\beta f)$
has at least two extreme points, for $\beta>0$ sufficiently large.
It is natural to expect that the set of extreme points
$\mathrm{ex}(\mathscr{G}^{*}(\beta f))= \{\mu_{+}^{\beta},\mu_{-}^{\beta} \} $,
where $\mu_{\pm}^{\beta}$ are the Thermodynamic Limits with plus and minus boundary conditions,
respectively. This would follow from a modification of the argument of the Aizenman-Higuchi Theorem
for an analogous model on the lattice $\mathbb{Z}$.
Once this is established, the remaining task would be proving
the existence of a harmonic function for $\mathbb{L}$ acting on $L^1(\nu)$, where
$\nu = \frac{1}{2}\mu_{+}^{\beta}+\frac{1}{2}\mu_{-}^{\beta}$.

Several results presented here can be easily extended to the non-compact alphabets.
The only obstacle is the existence of at least one conformal measure in $\mathscr{G}^{*}$.
The existence of such conformal measures, under some regularity assumptions on the potential,
has been shown in many papers. See, for example,
\cite{MR3568728,MR4026981,MR1853808,MR3190215,MR1738951}. If we consider the non compact setting 
of \cite{MR3922537}, which allows some low regular potentials, one can use the results about the Martin Boundary developed in \cite{MR3922537} which characterize extreme conformal measures as
the directions of escape to infinity of typical orbits.


The double transpose $\mathbb{L}^{**}$ of the extension $\mathbb{L}$, on
the bidual of $L^1(\nu)$, was considered in \cite{CER17}. The authors prove the
existence of a positive (in the Banach lattice sense) eigenvector for this extension, for any continuous potential.
This can be seen as a result on the existence of a harmonic function, but in a weaker sense.
If we find sufficient conditions for such eigenvectors
to be in the image of the Jordan canonical map,
then we potentially have consequences on Theorem \ref{th:eigenspace-dimension} and
therefore some new information on the dimension of the SHF.

On the FCLT in Section \ref{FCLT}, very recently, another version
of this theorem was obtained in \cite{gallesco2020mixing} 
in the context of $g$-measures on infinite countable 
alphabets, and for non-local observables with non-summable correlations. 
By taking the one-point compactification of $\mathbb{N}$ 
and using the same technique employed in \cite{MR3377291} 
to localize the support of the conformal measure away from the infinite,
the FCLT in \cite{gallesco2020mixing} can also be partially 
recovered in our setting. It seems to be possible to combine the
ideas of the present paper and the ones in \cite{gallesco2020mixing}
to obtain another version of a FCLT for non-local observables on uncountable and non-compact
alphabets.

\break   

\appendix
\section{Technical Results}\label{sec-supp-extension}

\subsection{The Extension of the Transfer Operator}

In this section we introduce the notion of  extension of $\mathscr{L}$
and provide sufficient conditions to its existence.

\begin{lemma}[Embedding Lemma]\label{lema-unicidade}
	If $\nu\in \mathscr{M}_{1}(X)$ is such that $\mathrm{supp}(\nu)=X$,
	then $\pi:C(X)\subset \mathcal{L}^1(\nu)\to L^1(\nu)$
	is a linear injective map.
\end{lemma}

\begin{proof} 
	If $\varphi,\psi \in C(X)$ are distinct, then there is a point $x\in X$, $r>0$ such that
	$\inf_{y\in B(x,r)} |\varphi(y)-\psi(y)|>0$.
	From the hypothesis it follows that $\nu(B(x,r))>0$
	and so $\pi(\varphi)\neq \pi(\psi)$.
\end{proof}

\begin{remark}
	It worth to mention that $L^1(\nu)$ can be, in some sense,
	a much smaller space than $C(X)$.
	An extreme example is obtained by taking $\nu=\delta_{x}$,
	the Dirac measure concentrated on $x\in X$.
	In this case $\dim_{\mathbb{R}} L^{1}(\nu)=1$,
	while $\dim_{\mathbb{R}}C(X)=\infty$. Therefore
	the restriction of the natural map $\pi$ to $C(X)$ can not be injective.
	In this case the support of $\nu$ is a singleton.
\end{remark}

\begin{definition}\label{def-ext-CC-to-L1L1}
	Let $\nu\in\mathscr{M}_{1}(X)$. We say that a bounded positive linear operator
	$\mathbb{L}:L^{1}(\nu)\to L^1(\nu)$ is an extension of a transfer operator
	$\mathscr{L}:C(X)\to C(X)$ if the vector space
	$C(X)$ embeds in $L^1(\nu)$ and for any $\varphi\in C(X)$ we have
	$\mathbb{L}[\varphi]_{\nu}\cap C(X) = \{\mathscr{L}\varphi\}$.
\end{definition}

Since the support of a measure on a separable topological space is a closed set it follows that
$U\equiv X\setminus\textrm{supp}(\nu)$ is always an open set. If $U$ is not empty, then
the set $\mathbb{L}[\varphi]_{\nu}\cap C(X)$ has an infinite number of elements and so
the conditions of the above definition are not satisfied.

\begin{theorem}\label{Teo-Extension-L1}
	If $\mathrm{supp}(p)=E$ and $\nu\in\mathscr{G}^{*}$ then
	$\mathscr{L}:C(X)\to C(X)$ extends to a
	bounded positive linear operator $\mathbb{L}:L^1(\nu)\to L^1(\nu)$.
	Moreover, the operator norm of this extension is  $\|\mathbb{L}\|_{\mathrm{op}}=\rho(\mathscr{L})$.
\end{theorem}

\begin{proof}
	Since $\mathrm{supp}(p)=E$  and $\nu\in\mathscr{G}^{*}$ we can apply  Theorem \ref{Teo-EP-fully-supp}
	to conclude that $\mathrm{supp}(\nu)=X$ and that $C(X)$ embeds in $L^1(\nu)$.
	Consider the linear operator $\widetilde{\mathbb{L}}:\pi(C(X))\to L^{1}(\nu)$ defined by
	$\widetilde{\mathbb{L}}\pi\varphi\equiv \pi\mathscr{L}\varphi$,
	for all $\varphi\in C(X)$. Since $\pi:C(X)\to L^1(\nu)$ is injective it follows that $\widetilde{\mathbb{L}}$
	is well-defined.
	From the definition of $\mathscr{G}^{*}$ we get that
	\begin{align*}
		\|\widetilde{\mathbb{L}}\pi\varphi \|_{L^1(\nu)} =\int_{X} |\widetilde{\mathbb{L}}\pi\varphi| \, d\nu
		=  \int_{X} |\pi\mathscr{L}\varphi| \, d\nu
		\leqslant \int_{X} \mathscr{L}|\varphi| \, d\nu
		& \leqslant \rho(\mathscr{L}) \int_{X}|\varphi| \, d\nu
		\\
		& =  \rho(\mathscr{L}) \|\pi\varphi \|_{L^1(\nu)}.
	\end{align*}

	The linearity of $\widetilde{\mathbb{L}}:\pi(C(X))\to L^{1}(\nu)$ and the
	above inequality imply that $\widetilde{\mathbb{L}}$
	is a Lipschitz function. Recalling that $\pi (C(X))$
	is dense subset of  $L^1(\nu)$ it follows from
	a classical result on real Analysis that $\widetilde{\mathbb{L}}$ has a bounded
	linear extension, which will be called $\mathbb{L}$, defined on the whole space $L^1(\nu)$.
	From  construction the identity
	$\mathbb{L}[\varphi]_{\nu}\cap C(X) = \{\mathscr{L}\varphi\}$ holds for any
	continuous test function $\varphi$.
	To finish the proof, we observe that the identity  $\|\mathbb{L}\|=\rho(\mathscr{L})$
	is an immediate consequence of the above inequalities
	being attained by $\varphi\equiv 1$ and the  positivity of $\mathbb{L}$.
\end{proof}

Note that Theorem \ref{Teo-Extension-L1} can be generalized to $L^{p}(X,\mathscr{B}(X),\nu)$
for any $p\in ([1,+\infty]$. The main difference is that for $p> 1$
the spectral radius $\rho(\mathscr{L})$ will not coincide with the operator norm.

\subsection{Integral Representation of the Extensions}\label{sec-int-representations}

Although the abstract argument in the last section is simple and direct,
it is not possible to conclude from it how to concrete represent
the action of the operator $\mathbb{L}$ on an arbitrary element of $L^1(\nu)$.
It is natural to expect, for example, that for any $B\in\mathscr{B}(X)$ the following holds.
If $\psi$ is a function in $\mathcal{L}^1(\nu)$ such that
$[\psi]_{\nu}= \mathbb{L}1_{B}$ then it is natural to ask whether
\begin{equation}\label{eq-int-rep-squareL}
	\psi(x) = \int_{E} \exp(f(ax))1_{B}(ax)\, dp(a) \ \nu-a.e..
\end{equation}
The answer to this question is positive if the dynamics is not singular  with
respect to the conformal measure $\nu$.
Actually  this is usually  the most used condition in most of the works that
consider the extensions of the classical Ruelle transfer operator to the Lebesgue spaces  $L^p(\nu)$,
with $1\leqslant p<+\infty$.

\medskip

\textbf{The Hypothesis H1.}
Let $\nu\in \mathscr{M}_{1}(X)$ be an arbitrary Borel probability measure on $X$.
By using the product structure of $X$ we can also consider the product measure
$p\times \nu$ as an element of	$\mathscr{M}_{1}(X)$,
which is defined on the cylinder sets in a natural way.
We will say that a pair $(p,\nu)$, where $\nu\in \mathscr{M}_{1}(X)$
and $p\in \mathscr{M}_{1}(E)$,
satisfies the hypothesis \eqref{hip1} if
\begin{equation}\tag{H1}\label{hip1}
	\exists K>0 \ \text{such that} \quad
	(p\times\nu)(B)\leqslant K\nu(B), \ \forall B\in\mathscr{B}(X).
\end{equation}

\begin{proposition}\label{prop-extensao-l1}
	Suppose that $(p,\nu)$ is a pair satisfying the hypothesis \eqref{hip1}
	then the transfer operator $\mathscr{L}:C(X)\to C(X)$ can be naturally extended to
	a positive linear transformation $L:\mathrm{dom}(L)\subset \mathcal{L}^{1}(\nu)\to\mathcal{L}^{1}(\nu)$
	given by
	\begin{align}\label{eq-Lphi}
		L\varphi(x)\equiv \int_{E}\exp(f(ax))\varphi(ax)\, dp(a), \quad \forall x\in X
	\end{align}
	and moreover
	\begin{equation}\label{eq-bound-L}
		\int_{X} |L\varphi|\, d\nu
		\leqslant
		Ke^{\|f\|_{\infty}} \int_{X}|\varphi|\, d\nu, \quad \forall \varphi\in \mathrm{dom}(L).
	\end{equation}
\end{proposition}

\begin{proof}
	The set $\mathrm{dom}(L)$ is the subset of all $\varphi\in \mathcal{L}^{1}(\nu)$
	for which the following expression makes sense for all $x\in X$
	\[
	L\varphi(x)\equiv \int_{E}\exp(f(ax))\varphi(ax)\, dp(a).
	\]
	This is clearly a linear subspace of $\mathcal{L}^{1}(\nu)$, and furthermore
	contains $C(X)$. Of course, $L\varphi \in M(X,\mathscr{B}(X))$. To show that $L\varphi \in
	\mathcal{L}^{1}(\nu)$
	and the validity of the inequality \eqref{eq-bound-L}
	it is enough to use the hypothesis \eqref{hip1}. In fact, for any $\varphi\in \mathrm{dom}(L)$,
	follows from \eqref{hip1} and elementary results of the Lebesgue integral that it
	\[
	\int_{X} |L\varphi|\, d\nu
	\leqslant
	\int_{X}\int_{E} \exp(f(ax))|\varphi(ax)|\, dp(a) d\nu(x)
	\leqslant
	Ke^{\|f\|_{\infty}} \int_{X}|\varphi|\, d\nu.
	\qedhere
	\]
\end{proof}

\bigskip

The  inequality \eqref{eq-bound-L} implies, in particular, that if
$\varphi=\psi$ $\nu$-a.e. are bounded measurable functions, then
that  $L\varphi=L\psi$ $\nu$-a.e. which means that $L$ preserves
$\nu$-equivalence classes. Therefore $L$ induces a bounded linear
operator  $\widetilde{\mathbb{L}}$ defined on a dense subset of $L^1(\nu)$.
By arguing as before we can extended $\widetilde{\mathbb{L}}$ to whole
$L^1(\nu)$ and this extension is precisely the operator
$\mathbb{L}:L^1(\nu)\to L^1(\nu)$ provided by Theorem \ref{Teo-Extension-L1}.

\begin{remark}
	For a general $\varphi\in \mathcal{L}^1(\nu)$ and $x\in X$
	the identity \eqref{eq-Lphi} may not be well-defined,
	even assuming \ref{hip1}. On the other hand, this assumption is enough to ensure that
	the rhs above is well-defined $\nu$-a.e..
\end{remark}

\begin{proposition}\label{mu_f satisfies C1}
	Let $\nu\in \mathscr{G}^{*}$ be a conformal measure
	and $p$ the a priori measure used to define $\mathscr{L}$.
	Then the pair $(p,\nu)$ satisfies the hypothesis \eqref{hip1}.
	As a consequence if the a priori measure $p$ has full support, then
	$\mathbb{L}$ has an integral representation as in \eqref{eq-int-rep-squareL}.
\end{proposition}

\begin{proof}
	The goal is to prove inequality in
	\eqref{hip1} for every Borel set
	$B\in \mathscr{B}(X)$.
	We first show its validity for a family of rectangles
	\[
	\mathscr{R}=\{U\times V: U \subseteq E \text{ and } V \subseteq X\  \text{are open sets}\}.
	\]
	
	Let $B \in \mathscr{R}$ of the form $B = U \times V$. Since $U$ is open in
	$E$, there is an increasing sequence of continuous functions $\psi_n: E
	\rightarrow [0,1]$ such that, for every $n\in \mathbb{N}$, $\psi_n \uparrow
	1_{U}$ pointwisely and, therefore, in $L^1(p)$. Similarly, there
	is an increasing sequence of continuous functions $\phi_n: X \rightarrow
	[0,1]$ (Urysohn functions) such that $\phi_n \uparrow 1_{V}$ again
	pointwisely and in $L^1(\nu)$. Therefore for any $x\in X$ we have that
	$\Psi_n(x) \equiv \psi_n(x_1) \phi_n(\sigma(x)) \uparrow 1_B(x)$.
	Clearly $\Psi_n\in C(X)$ and we have
	\begin{align*}
		\nu(B)
		& = \int_X 1_{B} d\nu \geqslant \int_X \Psi_n d\nu
		\\
		& = \frac{1}{\rho(\mathscr{L})}\int_X \Psi_n \, d[\mathscr{L}^{*}\nu]
		= \frac{1}{\rho(\mathscr{L})}\int_X \left( \mathscr{L} \Psi_n\right)\, d \nu
		\\
		& = \frac{1}{\rho(\mathscr{L})} \int_X \left[ \int_E
		\exp \left(f(ay) \right) \Psi_n(ax)\, dp(a) \right] d \nu(x)
		\\
		& = \frac{1}{\rho(\mathscr{L})}
		\int_X \int_{E} \exp \left( f(a x) \right) \psi_n(a) \phi_n(\sigma(ax)) \ dp(a) d \nu(x)
		\\
		& \geqslant \frac{e^{-\|f\|_\infty}}{\rho(\mathscr{L})} \int_{E} \psi_n(a)\, dp(a) \int_X \phi_n(x)\, d\nu(x).
	\end{align*}
	
	Thus, taking the limit when $n\rightarrow \infty$, one can conclude that
	\begin{equation*}
		\nu(B) \geqslant  \frac{e^{-\|f\|_\infty}}{\rho(\mathscr{L})}  p(U) \nu(V)
		=  \frac{e^{-\|f\|_\infty}}{\rho(\mathscr{L})} (p \times \mu)(B).
	\end{equation*}
	That is, inequality in \eqref{hip1} holds for any
	open rectangle and $K=\rho(\mathscr{L})e^{\|f\|_{\infty}}$.
	
	\medskip
	Since  the inequality in \eqref{hip1} holds for any element of
	$\mathscr{R}$ (which generates the Borel sigma-algebra $\mathscr{B}(X)$),
	it would be natural to expect that the same should be true for every Borel set of $X$.
	This is actually true,	but a careful argument is required to give a rigorous proof of this fact.
\end{proof}

To complete the proof we show first that the rectangles
of open sides approximate the rectangles of measurable
sides. Next we use that the family of finite disjoint unions of
rectangles with measurable sides form an algebra and conclude
by applying the Carathéodory  Extension Theorem.

The following proposition summarizes what was discussed on the last paragraphs.
This should be a very well known result, and we only prove it here because
we do not found a precise reference for this inequality.

\begin{proposition}
	Let $E$ and $F$ be two compact metric spaces and
	$\mu, \nu$ two Borel measures on
	the product space $(E \times F,\mathscr{B}(E\times F))$.
	If $\nu(U\times V) \leqslant \mu(U\times V)$
	for every open rectangle $U\times V$,
	then $\nu(B) \leqslant \mu(B)$ for every $B \in \mathscr{B}(E\times F)$.
\end{proposition}

\begin{proof}
	The first step towards this generalization is to approximate an arbitrary  rectangle $R = C\times D$ with
	measurable sides, i.~e., $R \in \mathscr{B}(E \times F)$  by a
	sequence of open rectangles.
	
	Since every Borel measure on a metric space is regular \cite{MR2169627}, the set function
	$\nu(C\times \cdot)$  defines a regular
	measure on $(F, \mathscr{B}(F))$. Hence, for every $\epsilon = 1/n$, it is
	possible to find an open set $V_n \supseteq D$ such that
	$\nu(C \times V_n) \leqslant \nu(C \times D) + 1/n$.
	Again, by regularity of $\nu(\cdot \times V_n)$, it is possible to find $U_n$
	open such that
	$\nu(U_n \times V_n) \leqslant \nu(C \times V_n) + 1/n$.
	Piecing  together the last two inequalities, we get that
	$\nu(U_n \times V_n) \leqslant \nu(C \times D) + 2/n$.
	This construction gives a sequence of open rectangles $(U_n \times V_n)$ which
	approximates $(C \times D)$ from above and it is such that $\nu(C\times D) =
	\lim\nu(U_n \times V_n)$.
	
	Using the above result for open rectangles we get that
	\[
		\nu(R)
		=
		\inf_{\substack{U\times V \subseteq X \text{ open} \\ R \subseteq U\times V}} \nu(U\times V)
		\geqslant
		\inf_{\substack{U\times V \subseteq X \text{ open} \\ R \subseteq U\times V}} \mu(U\times V)
		\geqslant
		\inf_{\substack{W \subseteq X \text{ open}         \\ R \subseteq W}} \mu(W)
		=
		\mu(R).
	\]
	This means that the desired inequality holds for every mensurable rectangle
	$R=C\times B$.
	
	It is clear that, if the above inequality holds separately for two disjoint
	measurable rectangles $R_1$ and $R_2$, it also holds for their union
	$R_1\cup R_2$ and more generally for any finite pairwise disjoint union of rectangles.
	Recall that the family $\mathscr{C}$
	of unions of pairwise disjoint mensurable rectangles forms an algebra of
	sets.
	
	From the last paragraph we conclude that $\mu|_\mathscr{C}\leqslant \nu|_\mathscr{C}$.
	Therefore the outer-measures associated to them will satisfy
	$(\mu|_\mathscr{C})^{*}\leqslant (\nu|_\mathscr{C})^{*}$.
	Since $\mu|_\mathscr{C}$ and $\nu|_\mathscr{C}$ are countable-additive pre-measures it follows from
	Carathéodory's Extension Theorem that $\mu=(\mu|_\mathscr{C})^{*}\leqslant (\nu|_\mathscr{C})^{*}=\nu$
	on $\sigma(\mathscr{C})=\mathscr{B}(E\times F)$.
\end{proof}

\subsection{The Spectrum of the Extended Operator}

In this section we obtain the spectrum of the extensions $\mathscr{L}$. 
The result is similar to the one known for finite alphabets  see, for example, \cite{MR2656475}.
However, its generalization for uncountable alphabets, presented here, is new. 
Before proceed, we should remind that when talking about the spectrum of the extension of the transfer 
operator $\mathbb{L}$, we are actually referring to the spectrum of its 
standard complexification, but for the sake of simplicity we 
will keep the same notation for both operators.

\begin{proposition}\label{prop-spectrum}
	Let $f$ be a general continuous potential and suppose that the a priori measure 
	satisfies the full support condition $\mathrm{supp}(p)=E$. Let $\nu\in\mathscr{G}^{*}$
	an arbitrary conformal measure and $\mathbb{L}:L^1(\nu)\to L^1(\nu)$ the extension 
	of the transfer operator associated to the potential $f$. 
	Then 
	$
	\mathrm{spec}(\mathbb{L}) = \{\lambda\in\mathbb{C}: |\lambda| \leqslant  \rho(\mathscr{L})\}
	$.  
\end{proposition} 

\begin{proof}
	Without loss of generality we can assume that $\rho(\mathscr{L})=1$. 
	Therefore the spectral radius of $\mathbb{L}:L^1(\nu)\to L^1(\nu)$ is also equal to one.
	Since $\mathrm{spec}(\mathbb{L}) = \mathrm{spec}(\mathbb{L^*})$ and the spectrum of 
	a bounded operator is a closed subset of the complex plane, it is enough to show that 
	$\{\lambda\in \mathbb{C}: |\lambda|<1\}\subset \mathrm{spec}(\mathbb{L}^{*})$. And 
	this is a consequence of the operator $\mathbb{L}^{*}-\lambda\mathrm{I}$ to be not onto whenever $|\lambda|<1$,
	hence it is not be invertible.  
	
	The main idea is to show that $\mathrm{Im}(\mathbb{L}^{*}-\lambda\mathrm{I})$ can not contains an 
	essentially bounded measurable function
	which is positive on a set $A$, with $\nu(A)>0$ and $\sigma(A)=X$; and identically 
	zero on a set $B$, which is disjoint from $A$ and also have positive measure $\nu(B)>0$.     
	
	Let us first construct the sets $A$ and $B$. They can be chosen as two disjoint open cylinder sets of the following
	form.  
	Take two distinct points $a,b\in E$ and $0<r < (1/2)d_{E}(a,b)$. 
	Define $A\equiv  B_{E}(a,r)\times E^{\mathbb{N}}$ and $B\equiv B_{E}(b,r)\times E^{\mathbb{N}}$.
	Since $A$ and $B$ are open sets of $X$ it follows from Theorem \ref{Teo-EP-fully-supp} that
	$\nu(A),\nu(B)>0$. By construction $\sigma(A)=X$ and $A\cap B = \emptyset$. 
	
	Let $\lambda\in\mathbb{C}$ be such that $|\lambda|<1$ and suppose by contradiction  
	that there is some complex function 
	$\psi = |\psi|\exp(i\arg(\psi))$
	such that 
	\[
	\mathbb{L}^{*}\psi -\lambda \psi = \mathds{1}_{B}.
	\]
	Since $\nu(B)>0$ it follows that $\psi$ can not be identically zero. 
	From Proposition \ref{prop-extensao-l1} we have that $\mathbb{L}^{*}\psi = \psi\circ\sigma$, $\nu$-almost everywhere. 
	Multiplying the above equation by $\mathds{1}_{A}$ we obtain the following identity
	\[
	\mathds{1}_{A}|\psi\circ \sigma|\exp(i\arg(\psi\circ\sigma)) -\lambda \mathds{1}_{A}|\psi|\exp(i\arg(\psi))= 0,
	\quad \nu-a.e.
	\]
	Therefore there is a measurable subset $X'\subset X$ such that $\nu(X')=1$ and the above equality
	holds for every $x\in X'$. By taking the modulus on the last expression and after the essential supremum we
	get
	\[
	\esssup_{x\in X'\cap A} |\psi\circ \sigma(x)|
	\leqslant |\lambda| \esssup_{x\in X'\cap A}|\psi(x)|
	\leqslant |\lambda|\esssup_{x\in X}|\psi(x)| 
	\equiv |\lambda|\|\psi\|_{\infty}.
	\]
	By the definition of $A$, we have $\sigma(X'\cap A)=\sigma(X')$. 
	From Proposition \ref{mu_f satisfies C1} it follows that $\sigma(X')$ is contains a set of
	$\nu$-measure one. Therefore it follows from the definition of 
	essential supremum that the left hand side above
	is equal to $\|\psi\|_{\infty}$. 
	But, this implies that $1\leqslant |\lambda|$ which is an absurd. 
\end{proof}

\section*{Acknowledgments}
The authors thanks Aernout van Enter, Artur Lopes, Daniel Takahashi and Paulo Varandas for their helpful comments, suggestions and references.

\bibliographystyle{plain}
\bibliography{references}

\end{document}